\documentclass[11pt,letterpaper]{amsart}
\usepackage{amssymb,amsmath,latexsym,amscd,amsfonts}
\usepackage{mathrsfs}
\usepackage{mathtools}
\usepackage{hyperref}
\usepackage[svgnames]{xcolor}
\hypersetup{colorlinks,breaklinks,
            linkcolor=DarkBlue,urlcolor=DarkBlue,
            anchorcolor=DarkBlue,citecolor=DarkBlue}
\usepackage{tikz}
\usepackage[normalem]{ulem}
\usepackage{float}

\usepackage{psfrag}
\usepackage{sidecap}
\usetikzlibrary{decorations.pathreplacing,shapes.geometric}
\usetikzlibrary{calc,positioning}
\usepackage[verbose,letterpaper,tmargin=.7in,bmargin=.7in,lmargin=.7in,rmargin=.7in]
{geometry}

\newcommand{\grc}[1]{\raisebox{-1.3cm}{\includegraphics[height=3cm]{TP#1.pdf}}}

\newtheorem{theorem}{Theorem}[section]

\newtheorem{conjecture}[theorem]{Conjecture}
\newtheorem{proposition}[theorem]{Proposition}
\newtheorem{corollary}[theorem]{Corollary}
\newtheorem{lemma}[theorem]{Lemma}

\newtheorem{fact}[theorem]{Fact}

\theoremstyle{definition}
\newtheorem{definition}[theorem]{Definition}

\newtheorem*{acknowledgement}{Acknowledgements}

\newtheorem{remark*}[theorem]{}
\theoremstyle{remark}
\newtheorem{remark}[theorem]{Remark}

\newcommand{\F}{\mathfrak{F}}

\def\ss{\subset}

\DeclareMathOperator*{\dprime}{\prime \prime}

\begin{document}
\title{An angle between intermediate subfactors and its rigidity}

\author{Keshab Chandra Bakshi}
\address{The Institute of Mathematical Sciences, Chennai, India}
\email{bakshi209@gmail.com}
\author{Sayan Das}
\address{Department of Mathematics\\ The University of Iowa}
\email{sayan-das@uiowa.edu}
\author{Zhengwei Liu}
\address{Department of Mathematics and Department of Physics\\ Harvard University}
\email{zhengweiliu@fas.harvard.edu}
\author{Yunxiang Ren}
\address{Department of Mathematics\\ University of the Tennessee}
\email{renyunxiang@gmail.com}

\begin{abstract}
We introduce a new notion of angle between intermediate subfactors and prove various interesting properties of the angle and relate it with the Jones' index. 
We prove a uniform 60 to 90 degree bound for the angle between minimal intermediate subfactors of a finite index irreducible subfactor. From this rigidity we can bound the number of minimal (or maximal) intermediate subfactors by the kissing number in geometry. As a consequence, the number intermediate subfactors of an irreducible subfactor has at most exponential growth with respect to the Jones index. This answers a question of Longo published in 2003.

\end{abstract}

\maketitle
\section{Introduction}
Jones pioneered the study of modern subfactor theory in his seminal paper (\cite{Jo1}). He showed that the indices of subfactors of type II$_1$ lie in the set
$
\{4\cos^2(\frac{\pi}{n}):n\geq3\}\cup[4,+\infty].
$
The study of intermediate subfactors $N \subset P \subset M$ for a finite index inclusion of $II_1$ factors plays an important role in understanding the theory of subfactors. (See \cite{Bi1}, \cite{BiJo1} for some early motivating results in this direction). 
We denote by $\mathcal{L}(N\subset M)$ the set of all intermediate von Neumann subalgebras for the subfactor $N\subset M$. The set $\mathcal{L}(N\subset M)$ forms a lattice under the two operations
$P\wedge Q= P\cap Q$ and $P\vee Q=\{P\cup Q\}^{\dprime}$. The lattice structure of von Neumann subalgebras was first studied by
Murray and von Neumann in \cite{MaVo}.  If $N\subset M$ is irreducible, that is $N^{\prime}\cap M= \mathbb{C}$, then $\mathcal{L}(N\subset M)$ is exactly 
the lattice of intermediate subfactors. In this case, all intermediate subalgebras are automatically factors. 

The lattice of intermediate subfactors generalize the lattice of subgroups because of the following reasons:
Let $G$ be a finite group with an outer action on the II$_1$ factor $M$. 
 Then the intermediate subfactors of $M\subset M\rtimes G$ are 
 given by $M\rtimes H$, where $H$ is a subgroup of $G$. 
 This leads us to the study of the lattice $\mathcal{L}(N\subset M)$ inspired by various interesting questions in group theory. See \cite{GX,Xu3,Xu2,Xu,Pa} for some recent progress.

Watatani in \cite{Wa}, following previous work by Popa \cite{Po1}, obtained the following remarkable result:
\begin{theorem}\label{watatani}\cite{Wa}
 Let $N\subset M$ be an irreducible subfactor of type $II_1$ such that $[M:N] < \infty$. Then 
 the set $\mathcal{L}(N\subset M)$ is finite.
\end{theorem}

In the same paper  Watatani remarked that in this case we can regard an intermediate subfactor lattice as a ``quantization'' of continuous geometry, invented by von Neumann in \cite{Ne}, as
a continuous analogue of projective geometry. For a general finite index subfactor $N\subseteq M$ the set of all intermediate subfactors may  not be finite.  Even in the case when $N^\prime \cap M$
is abelian the set of intermediate subfactors may be infinite as shown in \cite{TW} (Theorem 5.4).  Thus finite index irreducible subfactors may behave very differently from non irreducible inclusions. 

Inspired by earlier works of Christensen and Watatani (see \cite{Cr1}, \cite{Wa}), Longo gave an explicit bound for the number of intermediate subfactors for irreducible subfactors in \cite{Lo}. 
He showed that the number of intermediate subfactors is bounded by $[M:N]^{2[M:N]^2}$.
Longo then asked if the number of intermediate subfactors can be bounded by $[M:N]^{[M:N]}$ (see discussions at the end of section 2.2 in \cite{Lo}). 
In this paper we answer this question positively by showing that:
\begin{theorem}\label{Thm:inter}
	Let $N \subseteq M$ be a finite index, irreducible subfactor. Then the number of intermediate subfactors $|\mathcal{L}(N\subset M)|$ is bounded by $9^{[M:N]}$
\end{theorem}
 Our bound improves the existing upper bound of the cardinality of the lattice $\mathcal{L}(N\subset M).$ We have improved Longo's 
 bound using purely planar algebraic machinery and as a consequence provide another proof of Theorem \ref{watatani}. To solve this problem of finding upper bound for the cardinality of $\mathcal{L}(N\subset M)$, our idea is to firstly focus on
 minimal intermediate subfactors.

Minimal (or by duality maximal) subfactors were extensively studied by Guralnick and Xu \cite{GX} inspired by Wall's conjecture  \cite{Wal61}: 
\begin{conjecture}[Wall's conjecture]
	For a finite group $G$, let $\max(G)$ denote the number of maximal proper subgroups of $G$. Then we have
	\begin{equation*}
	\max(G)\leq \vert G\vert.
	\end{equation*}
\end{conjecture}
A generalization of Wall's conjecture has been proposed in the context of subfactors (see Conjecture $1.1$ in \cite{GX}).
During the June 2012 AIM workshop, ``Cohomology bounds and growth rates", a counterexample was found to Wall's conjecture \cite{GHPS12}.

Our result shows that the number of minimal intermediate subfactors is at most exponential growth with respect to the index. 
We conjecture that the number has polynomial growth:
\begin{conjecture}
There are constants $c_1, c_2$, so that for any irreducible subfactor $N \subseteq M$  with finite index, the number of minimal intermediate subfactors is less than $c_2[M:N]^{c_1}$.
\end{conjecture}

Furthermore, we prove that the number of minimal intermediate subfactors is bounded by the kissing number $\tau_n$ of the $n$-dimensional sphere, where $n=\dim (N'\cap M_1)$. 
A straightforward estimate of the kissing number shows that $\tau_n < 3^n$. Therefore we get:
\begin{theorem}\label{Thm:min}
Suppose $N\subset M$ is a finite index, irreducible subfactor. Then the number of minimal intermediate subfactors is less than $3^{\dim (N'\cap M_1)}$.
\end{theorem}

We prove the above theorem by introducing a new angle (see Definition \ref{alpha}), denoted by $\alpha^N_M(P,Q)$, between intermediate subfactors $P$ and $Q$ of any finite index subfactor $N\subset M$. Our definition uses the $1$-$1$ correspondence between intermediate subfactors (of irreducible subfactor) and biprojections introduced in \cite{Bi1} (reformulated in planar algebraic terms - which we will actually use - as in \cite{La} and \cite{BiJo2}). The angle is also the Fourier dual of the correlation function. 
We prove the following rigidity result for the angle between minimal intermediate subfactors:

\begin{theorem}\label{Thm:angle}
If  $P,Q$ are two distinct minimal intermediate subfactors of a finite index, irreducible subfactor $N\subset M$, then $ \frac{\pi}{3}< \alpha^N_M(P,Q) \leq  \frac{\pi}{2}$. 
\end{theorem}

We can identify intermediate subfactors as unit vectors in the real vector space $(N' \cap M_1)_{s.a.}$ such that the angle between them are given by $\alpha$. Then Theorem \ref{Thm:min} follows from Theorem \ref{Thm:angle}. Iterating Theorem \ref{Thm:min}, we obtain Theorem \ref{Thm:inter}.

We also study the angle for intermediate subfactors of non-irreducible subfactors and show that $\alpha^N_M(P,Q) = \frac{\pi}{2}$ if and only if the quadruple
$$\begin{matrix}
Q &\subset & M \cr
\cup &\ &\cup\cr
N  &\subset & P
\end{matrix},$$
denoted by $(N,P,Q,M)$, is a commuting square, namely $E^M_P E^M_Q= E^M_Q E^M_P = E^M_N$. Commuting square is a central tool in subfactor theory. See for example \cite{JS,GHJ,Po2,Po3,Po4,Po5}, to name a few.
By Fourier duality, we define a dual angle $\beta^N_M(P,Q)$ (in Definition \ref{beta}), and show that $\beta^N_M(P,Q) = \frac{\pi}{2}$ if and only if the quadruple is a co-commuting square.

In general, the angles $\alpha^N_M(P,Q)$ and $\beta^N_M(P,Q)$ are different. 
Surprisingly, the following result holds:
\begin{theorem}
Suppose $P,Q$ are two distinct intermediate subfactors of a finite index subfactor $N\subset M$. If $[M:Q]= [P:N]$ and hence $[M:P] = [Q:N]$, then $\alpha^N_M(P,Q)=\beta^N_M(P,Q)$.
\end{theorem} 
When the equality holds, we call the quadruple $(N,P,Q,M)$ a {\it parallelogram}, and consider the angles $\alpha^N_M (P, Q)$ and $\beta^N_M(P, Q)$ as opposite angles of the parallelogram.
 
We further study the relation between angles and Pimsner-Popa bases and derive various equivalent conditions for a quadruple to become a commuting and/or co-commuting square (see Theorems \ref{q}, \ref{cs}, \ref{sym}). As a consequence, we recover the various equivalent 
conditions of `non-degenerate commuting square' by Popa in \cite{Po2}  (see Corollary \ref{po2}).

This paper is organized as follows. In \S \ref{Sec:preliminary} we recall some basic definitions and results from the theory of planar algebras. 
In \S \ref{Sec:angle} we define our angle and obtain various properties, mainly related to commuting squares. 
In \S \ref{Sec:rigidity} we prove the main rigidity result Theorem \ref{Thm:angle}.
In \S \ref{Sec:number} we estimate the number of intermediate subfactors and prove Theorems \ref{Thm:min} and \ref{Thm:inter}.

\begin{acknowledgement}
The first author would like to thank V.S.Sunder and Vijay Kodiyalam for various useful discussions.
The second author would like to thank Jesse Peterson for his encouragement and various helpful discussions.
The third author would like to thank Feng Xu for helpful discussions. 
All authors would like to thank Hausdorff Research Institute for Mathematics, Bonn, where this project had been started, for their kind hospitality.
The first author is supported by IMSc and HBNI.
The third author is supported in part by Templeton religion trust under grants TRT 0080 and TRT 0159.
\end{acknowledgement}

\section{Preliminary}\label{Sec:preliminary}
The central object in the subfactor theory is the standard invariant of a given subfactor. For a crash course in subfactor theory the reader is referred to the book \cite{JS}. A deep theorem by Popa says that the standard invariant completely determines strongly amenable subfactors \cite{Po2}. Moreover, Popa introduced standard $\lambda$-lattice as an axiomatization of the standard invariant in \cite{Po3} which completes Ocneanu's paragoroup axiomatization for subfactors of finite depth \cite{Ocn88}. Jones  subsequently introduced subfactor planar algebras as an axiomatization of the standard invariant of subfactors, which capture the topological properties \cite{Jo2}. We briefly recall the definition of planar algebras and the correspondence to the standard invariant of a subfactor in \cite{Jo2}.

A planar tangle is defined in $\mathbb{R}^2$. It consists of the following data:

(1) An output disc. In the interior there exist finitely many input discs. 

(2) Finitely many smooth strands in the interior of the output disc and the complement of the input discs which meet the boundaries of discs transversally. 

(3) As each boundary is partitioned into finitely many intervals, a $\$$ sign  is assigned to a distinguished interval (to indicate the relative position). 

(4) The connected components of $\mathbb{R}^2$ in the output disc are called regions. The regions admit a checkerboard coloring.

The planar algebra consists of a sequence of graded vector space ${P}_{n,\pm}$ which admit actions of planar tangles. Each planar tangle corresponds to a multilinear map between tensor products of the vectors spaces determined by the boundary conditions of the planar tangle. An example of a planar tangle is given below:
\begin{equation*}
T=\grc{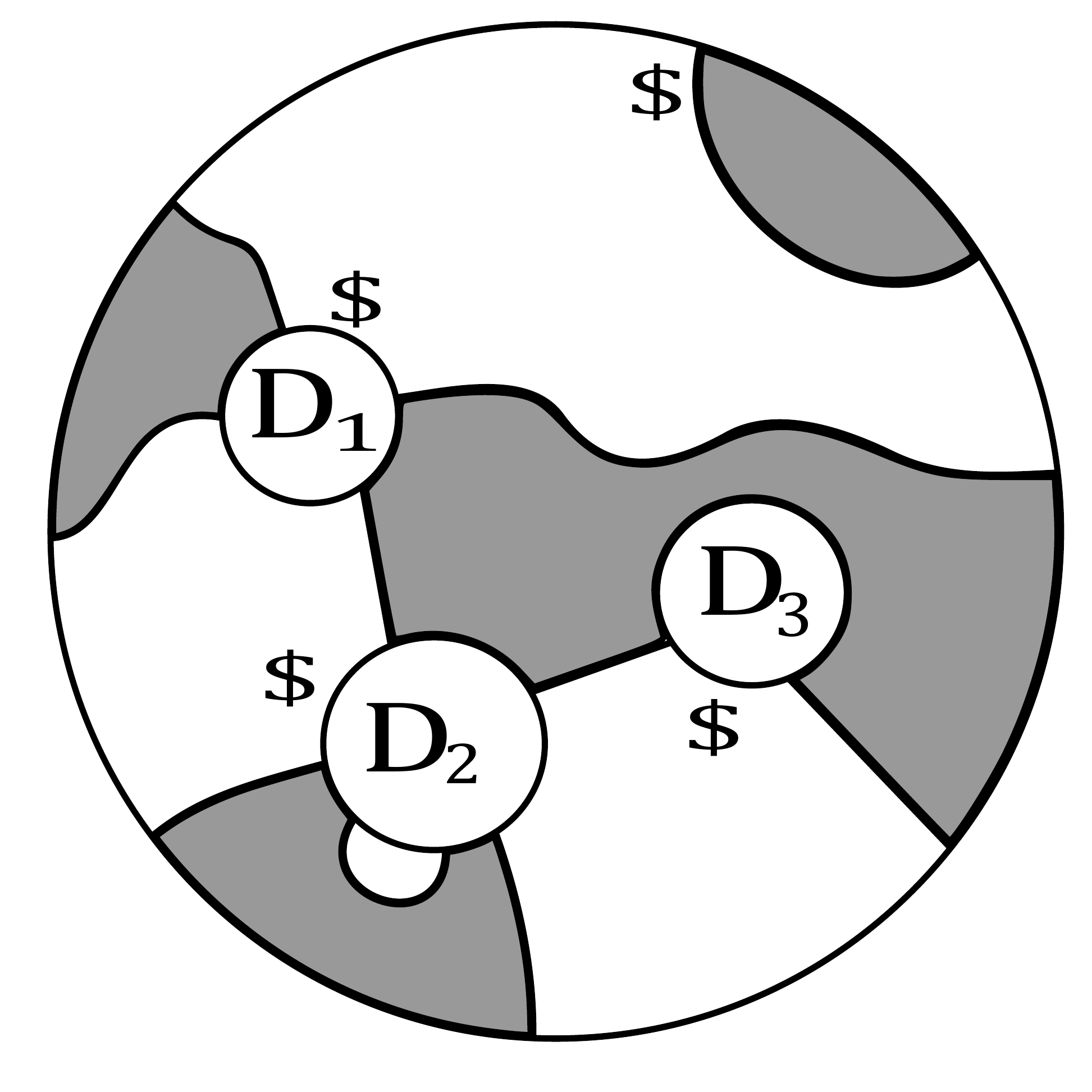}
\end{equation*}

The planar tangle $T$ corresponds to a multilinear map $Z_T: {P}_{2,+}\otimes {P}_{3,+}\otimes {P}_{1,+}\rightarrow {P}_{4,+}$, called the partition function. Composition of tangles/partition functions is demonstrated below. 
\begin{align*}
&T=\grc{Tangle1},~~S=\grc{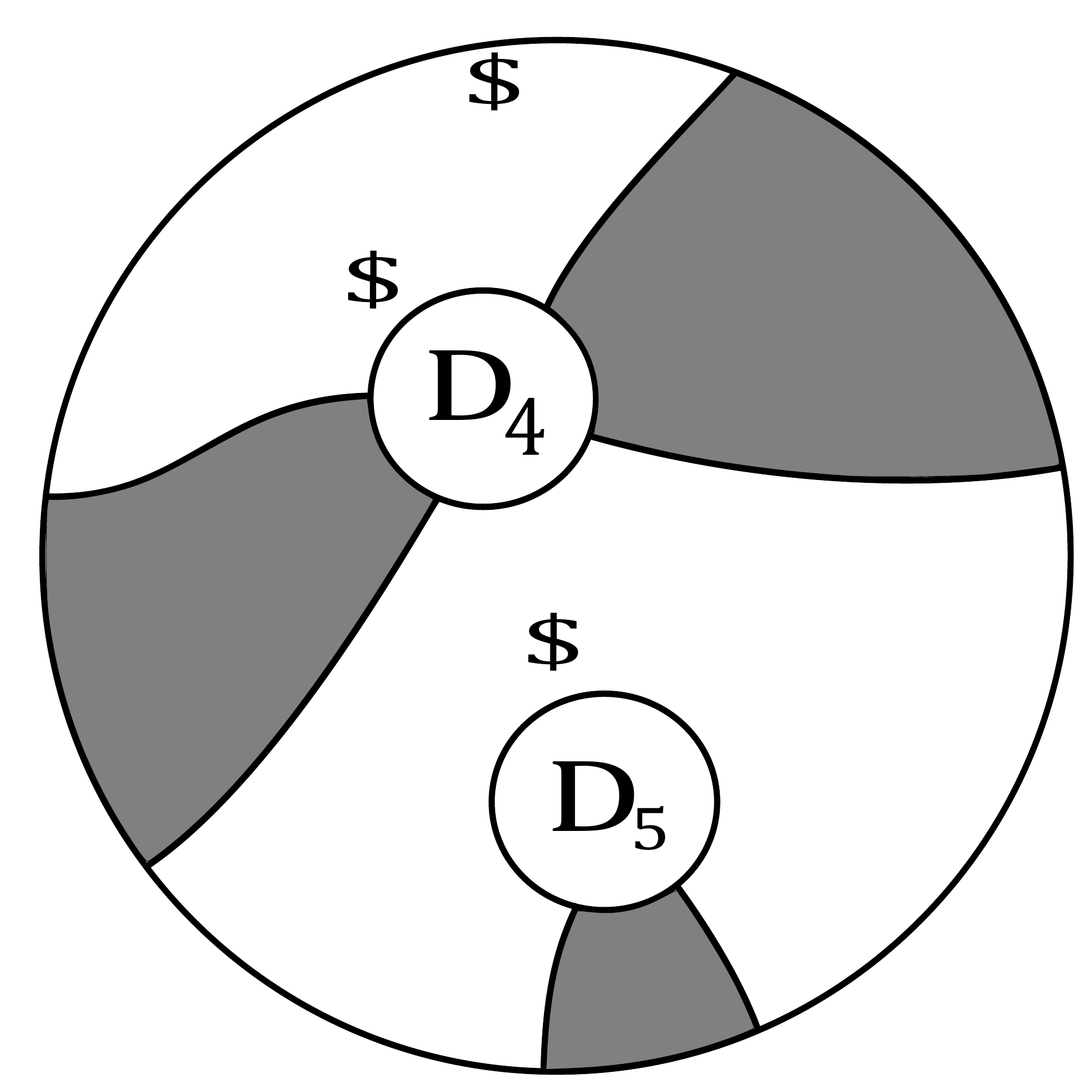},\\
&T\circ_{D_2}S=\grc{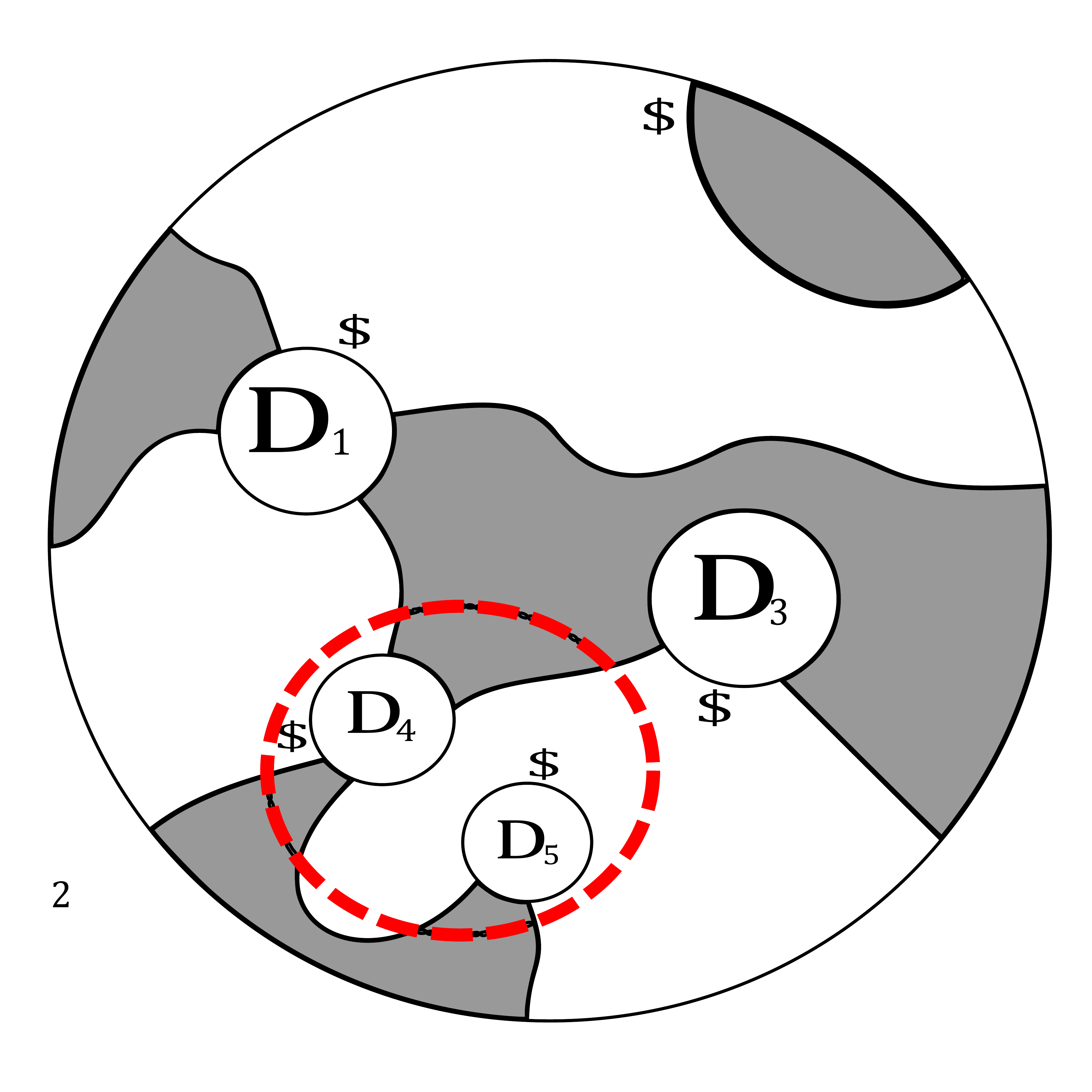}.
\end{align*}
The planar tangle action determines more structure on each vector space $P_{n,\pm}$. To be more specific, we have the following tangle:

Multiplication: (We omit the shading here)
\begin{equation*}
\begin{tikzpicture}
\draw (0,0) circle [radius=1.8];
\draw (0,.9) circle [radius=.6];
\draw (0,-.9) circle [radius=.6];
\draw (0,1.5)--(0,1.8);
\draw (0,.3)--(0,-.3);
\draw (.707*.6,.707*.6+.9)--(.707*.6,1.75);
\draw (-.707*.6,.707*.6+.9)--(-.707*.6,1.75);
\draw (.707*.6,-.707*.6+.9)--(.707*.6,.707*.6-.9);
\draw (-.707*.6,-.707*.6+.9)--(-.707*.6,.707*.6-.9);
\draw (0,-1.5)--(0,-1.8);
\draw (.707*.6,-.707*.6-.9)--(.707*.6,-1.75);
\draw (-.707*.6,-.707*.6-.9)--(-.707*.6,-1.75);
\node at (-.6,.9) [left] {\$};
\node at (-.6,-.9) [left] {\$};
\node at (-1.8,0) [right] {\$};
\draw (.05,0)--(.35,0) [ultra thick, dotted];
\draw (.05,1.65)--(.35,1.65) [ultra thick, dotted];
\draw (.05,-1.65)--(.35,-1.65) [ultra thick, dotted];
\end{tikzpicture}
\end{equation*}

Following the notation of \cite{KodSun} we denote the so-called generating tangles  by $\{{\mathcal{E}}^k:k\geq 2\}, \{{(E^{\prime})}^k_k: k\geq 1\}$ and $\{E^k_{k+1},M_k, I^{k+1}_k : k\in Col\}$ which are called 
Jones Projection tangles, left conditional expectation tangles,(right) conditional expectation tangles, multiplication tangles and inclusion tangles  respectively.
As is usual, we will sometimes draw the discs as boxes and we may sometimes omit drawing the external disc. If from the context the shading is clear we will omit that also.

To be a subfactor planar algebra, we have some additional requirements:
\begin{itemize}
	\item $\dim(P_{0,\pm})=1;$
	\item $\dim(P_{n,\pm})<\infty$;
	\item For each $n\in\mathbb{N}$, the planar algebra $P_{n,\pm}$ admits an involution $*:P_{n,\pm}\rightarrow P_{n,\pm}$ and a positive-definite trace $tr$ defined by ${\delta}^{-n}$ times the following figure:
	\begin{equation*}
	\begin{tikzpicture}
	\draw (0,0) circle [radius=1.8];
	\draw (0,0) circle [radius=.9];
	\draw (-.707*.9,.707*.9) arc [radius=1, start angle=180, end angle=0];
	\draw (-.707*.9+2,.707*.9)--(-.707*.9+2,-.707*.9);
	\draw (-.707*.9,-.707*.9) arc [radius=1, start angle=-180, end angle=0];
	\draw (0,.9) arc [radius=.5, start angle=180, end angle=0];
	\draw (1,.9)--(1,-.9);
	\draw (0,-.9) arc [radius=.5, start angle=-180, end angle=0];
	\draw (1.05,0)--(2-.707*.9,0)[dotted, ultra thick];
	\node at (-.9,0) [left] {\$};
	\node at (-1.8,0) [right] {\$};
	\end{tikzpicture}
	\end{equation*}
\end{itemize}

Jones proved the following remarkable theorem, which provide the correspondence between subfactors and subfactor planar algebras:
\begin{theorem}\cite{Jo2} \label{jones}
 Let $N\ss M (=M_0){\ss}^{e_1} M_1 \ss \cdots {\ss}^{e_k} M_k {\ss}^{e_{k+1}}\cdots $
 be the tower of the basic construction associated to an extremal subfactor with $[M:N]= {\delta}^2<\infty.$
 Then there exists a unique subfactor planar algebra $P= P^{N\subseteq M}$ of modulus $\delta$ satisfying the following conditions:
 \begin{enumerate}
 \item $P^{N\subseteq M}_k= N^{\prime} \cap M_{k-1} \forall k\geq 1$-where this is regarded as an equality of $*$ algebras which is consistent with the inclusions on the  two sides;
 \item  $Z_{{\mathcal{E}}^{k+1}}(1)= \delta e_k \forall k\geq 1;$
 \item  $Z_{{(E^{\prime})}^k_k}(x)= \delta E_{M^{\prime}\cap M_{k-1}}\forall x\in N^{\prime}\cap M_{k-1}, \forall k\geq 1;$
  \item $Z_{E^k_{k+1}}(x)= \delta E_{N^{\prime}\cap M_{k-1}}(x) \forall x \in N^{\prime}\cap M_k$; and this is required to hold for all $k$ in Col, where for $k=0_{\bar{{+}}}$,
 the equation is interpreted as $Z_{E^{\bar{{+}}}_1}(x)= \delta tr_M(x) \forall x \in N^{\prime}\cap M.$
 \end{enumerate}
 Conversely (Theorem 4.3.1 of \cite{Jo2}), given any spherical $C^*$
planar algebra $P$, there exists an extremal subfactor
$N \ss M$ such that $P$ is equivalent to $P^{N\ss M}$ as planar algebras.
We will say that a planar algebra $P$ and a subfactor $N \ss M$ are
associated to one another if $P$ is equivalent to $P^{N\ss M}$.
\end{theorem}

\begin{remark}
If we remove the extremal condition for subfactors, then the left trace and the right trace on $N'\cap M$ could be different. 
An irreducible subfactor (i.e. $N'\cap M=\mathbb{C}$), is always extremal. 
\end{remark}

 It is well-known from earlier work of Bisch \cite{Bi1} that intermediate subfactors $Q$ of a finite index irreducible subfactor $N\subseteq M$ are in bijective correspondence with so-called {\it biprojections} $e_Q$ in $N' \cap M_1$.
More precisely, $e_Q$ is the projection from $L^2(M)$ onto $L^2(Q)$, which is also an $N$-$N$ bimodule map in $P^{(N\subseteq M)}_{2,+}$, still denoted by $e_Q$.
This result has been reformulated in  \cite{La} and \cite{BiJo2}   in planar algebraic terms:

\begin{theorem}\label{Bisch}\cite{Bi1,BiJo2,La}  Let $N \ss M$  be a finite index, irreducible subfactor of type ${\rm II}_1$, and $P^{(N\ss M)}$ be its planar algebra.
For any intermediate subfactor $Q$, $N \ss Q \ss M$, the biprojection $e_Q \in P^{(N\subseteq M)}_{2,+}$ satisfies the following relations:  
\begin{align*}
&\begin{tikzpicture}
\node at (-1,.25) {$(a)$};
\path [fill=lightgray] (.3,-.5) rectangle (.7,1);
\draw (.3,-.5)--(.3,1);
\draw (.7,1)--(.7,-.5);
\draw [fill=white](0,0) rectangle (1,.5);
\node at (0,.25) [left] {$\$$};
\node at (.5,.25) {$e_Q$};
\node at (1.1,.25)[right] {$=$};
\path [fill=lightgray] (2.3,-.5) rectangle (2.7,1);
\draw (2.3,-.5)--(2.3,1);
\draw (2.7,1)--(2.7,-.5);
\draw [fill=white] (2,0) rectangle (3,.5);
\node at (2.5,.25) {$e_Q$};
\node at (3,.25)[right]{$\$$};
\node at (5,.25) {$(b)$};
\path [fill=lightgray] (6.3,-.5)--(6.3,.6) arc [radius=.2,start angle=180,end angle=0]--(6.7,-.5);
\draw (6.3,-.5)--(6.3,.6);
\draw (6.7,-.5)--(6.7,.6);
\draw (6.7,.6) arc [radius=.2,start angle=0, end angle=180];
\draw [fill=white](6,0) rectangle (7,.5);
\node at (6,.25)[left] {$\$$};
\node at (6.5,.25) {$e_Q$};
\node at (7.5,.25) {$=$};
\path [fill=lightgray] (8.3,-.5)--(8.3,0) arc[radius=.2,start angle=180, end angle=0]--(8.7,-.5);
\draw (8.3,0) arc[radius=.2,start angle=180, end angle=0];
\draw (8.3,0)--(8.3,-.5);
\draw (8.7,0)--(8.7,-.5);
\end{tikzpicture}\\
&\begin{tikzpicture}
\node at (-1,.25) {$(c)$};
\path [fill=lightgray](.3,-.5) rectangle (.7,1);
\path [fill=lightgray] (.7,1)--(.7,.6) arc[radius=.3,start angle=180, end angle=0]--(1.3,.6)--(1.6,.6)--(1.6,1);
\path [fill=lightgray] (.7,-.5)--(.7,-.1) arc[radius=.3,start angle=180, end angle=360]--(1.3,.6)--(1.6,.6)--(1.6,-.1)--(1.6,-.5);
\draw (.3,-.5)--(.3,1);
\draw (.7,.6)--(.7,-.1);
\draw (.7,.6) arc[radius=.3, start angle=180, end angle=0];
\draw (1.3,.6)--(1.3,-.1);
\draw (.7,-.1) arc[radius=.3, start angle=180, end angle=360];
\draw [fill=white] (0,0) rectangle (1,.5);
\node at (0,.25)[left] {$\$$};
\node at (.5,.25) {$e_Q$};
\node at (1.8,.25) [right] {$=c$};
\path [fill=lightgray](3,1) rectangle (3.5,-.5);
\draw (3,1)--(3,-.5);
\node at (3,.25) [left]{$\$$};
\end{tikzpicture}\\
&\begin{tikzpicture}
\node at (-1,.25) {$(d)$};
\path [fill=lightgray] (.3,1)--(.3,-2)--(.7,-2)--(.7,-1.5) arc[radius=.3,start angle=180, end angle=90]--(2.2,-1.2) arc[radius=.3,start angle=90, end angle=0]--(2.5,-2)--(2.8,-2)--(2.8,1)--(2.5,1)--(2.5,-.5) arc[radius=.3,start angle=0,end angle=-90]--(1,-.8) arc[radius=.3,start angle=270,end angle=180]--(.7,1);
\draw (.3,1)--(.3,-2);
\draw (.7,1)--(.7,-.5);
\draw (.7,-.5) arc [radius=.3,start angle=180, end angle=270];
\draw (1,-.8)--(2.2,-.8);
\draw (2.2,-.8) arc[radius=.3,start angle=270, end angle=360];
\draw (2.5,-.5)--(2.5,1);
\draw (2.2,-1.2) arc[radius=.3,start angle=90, end angle=0];
\draw (2.2,-1.2)--(1,-1.2);
\draw (1,-1.2) arc[radius=.3,start angle=90,end angle=180];
\draw (.7,-1.5)--(.7,-2);
\draw (2.5,-1.5)--(2.5,-2);
\draw [fill=white] (0,0) rectangle (1,.5);
\draw [fill=white] (1.5,-.5) rectangle (2,-1.5);
\node at (0,.25)[left] {$\$$};
\node at (1.75,-1.5)[below] {$\$$};
\node at (.5,.25) {$e_Q$};
\node at (1.75,-1){$e_Q$};
\node at (3.5,-.5){$=$};
\path [fill=lightgray] (4.8,1)--(4.8,-2)--(5.2,-2)--(5.2,-.5) arc[radius=.3,start angle=180,end angle=90]--(6.7,-.2) arc[radius=.3,start angle=90,end angle=0]--(7,-2)--(7.3,-2)--(7.3,1)--(7,1)--(7,.5) arc[radius=.3,start angle=0, end angle=-90]--(5.5,.2) arc[radius=.3,start angle=270,end angle=180]--(5.2,1);
\draw (4.8,-2)--(4.8,1);
\draw (5.2,-2)--(5.2,-.5);
\draw (5.2,-.5) arc[radius=.3, start angle=180, end angle=90];
\draw (5.5,-.2)--(6.7,-.2);
\draw (6.7,-.2) arc[radius=.3,start angle=90, end angle=0];
\draw (7,-.5)--(7,-2);
\draw (7,1)--(7,.5);
\draw (7,.5) arc[radius=.3,start angle=0, end angle=-90];
\draw (6.7,.2)--(5.5,.2);
\draw (5.5,.2) arc[radius=.3,start angle=270, end angle=180];
\draw (5.2,.5)--(5.2,1);
\draw [fill=white] (4.5,-1.5) rectangle (5.5,-1);
\draw [fill=white] (6,.5) rectangle (6.5,-.5);
\node at (4.5,-1.25) [left] {$\$$};
\node at (6.25,-.5)[below] {$\$$};
\node at (5,-1.25) {$e_Q$};
\node at (6.25,0){$e_Q$};
\end{tikzpicture}
\end{align*}
with $c = [Q:N ][M:N]^{-1/2}$. 
Conversely, if an operator in $P^{(N\subseteq M)}_{2,+}$ satisfies the above relations, then it is $e_Q$ for some intermediate subfactor $Q$.
  \end{theorem}

The above relations (a)-(d) of $e_Q$ in the above theorem are often referred to {\it an exchange relation} which has been further studied in \cite{La,BiJo2,Liu}.

For $x,y \in N' \cap M_1$, the coproduct of $x$ and $y$ is defined as
\begin{equation*}
\begin{tikzpicture}
\path [fill=lightgray] (.3,1.05)--(2.2,1.05)--(2.2,.5)--(1.8,.5) arc[radius=.55,start angle=0, end angle=180]--(.3,.5);
\path [fill=lightgray] (.3,-.55)--(2.2,-.55)--(2.2,0)--(1.8,0) arc[radius=.55,start angle=0, end angle=-180]--(.3,0);
\draw (.3,-.55)--(.3,1.05);
\draw (2.2,-.55)--(2.2,1.05);
\draw (.7,0) arc [radius=.55, start angle=180, end angle =360];
\draw (.7,.5) arc [radius=.55, start angle=180, end angle =0];
\draw [fill=white] (0,0) rectangle (1,.5);
\draw [fill=white] (1.5,0) rectangle (2.5,.5);
\node at (.5,.25) {$x$};
\node at (2,.25) {$y$};
\node at (0.1,.25) [left]{$x*y=\$$};
\node at (1.6,.25) [left]{$\$$};
\end{tikzpicture}
\end{equation*}

The following result is due to Landau.(See also \cite{GJ} and further applications there).
\begin{theorem}[Landau]\label{Thm:Landau}
	Suppose $e_P$ and $e_Q$ are two biprojections in an irreducible subfactor planar algebra, then $\frac{1}{\delta tr(e_Pe_Q)}e_P* e_Q$ is a projection.
\end{theorem}
\begin{proof}
Using exchange relation  twice we get the following:

\begin{equation*}
\begin{tikzpicture}
\path [fill=lightgray] (.3,3.05)--(.3,-.55)--(1.25,-.55) arc [radius=.55,start angle=270, end angle=180]--(.7,.5) arc [radius=.55,start angle=180, end angle=90]--(1.25,1.45) arc [radius=.55,start angle=270, end angle =180]--(.7,2.5) arc [radius=.55,start angle=180, end angle=90];
\path [fill=lightgray] (2.5-.3,3.05)--(2.5-.3,-.55)--(2.5-1.25,-.55) arc [radius=.55,start angle=270, end angle=360]--(2.5-.7,.5) arc [radius=.55,start angle=0, end angle=90]--(2.5-1.25,1.45) arc [radius=.55,start angle=270, end angle =360]--(2.5-.7,2.5) arc [radius=.55,start angle=0, end angle=90];
\draw (.3,-.55)--(.3,3.05);
\draw (2.2,-.55)--(2.2,3.05);
\draw (.7,2.5) arc [radius=.55, start angle=180, end angle =0];
\draw (.7,0) arc [radius=.55, start angle=180, end angle =360];
\draw (.7,.5) arc [radius=.55, start angle=180, end angle =0];
\draw (.7,2) arc [radius=.55, start angle=180, end angle =360];
\draw [fill=white] (0,0) rectangle (1,.5);
\draw [fill=white] (1.5,0) rectangle (2.5,.5);
\draw [fill=white] (0,2) rectangle (1,2.5);z
\draw [fill=white] (1.5,2) rectangle (2.5,2.5);
\node at (0,.25) [left] {$\$$};
\node at (0,2.25) [left] {$\$$};
\node at (1.5,.25) [left] {$\$$};
\node at (1.5,2.25)[left] {$\$$};
\node at (.5,.25) {$e_P$};
\node at (.5,2.25) {$e_P$};
\node at (2,2.25) {$e_Q$};
\node at (2,.25) {$e_Q$};
\node at (3,1.25)[right] {$=$};
\path [fill=lightgray] (4.3,3.05)--(7.2,3.05)--(7.2,.5)--(6.8,.5)--(6.8,1.5) arc[radius=.3,start angle=0, end angle=90]--(6.5,2.2) arc [radius=.3,start angle=270, end angle=450]--(5,2.8) arc [radius=.3, start angle=90, end angle=270]--(5,1.8) arc [radius=.3,start angle=90, end angle=180]--(4.7,.5)--(4.3,.5);
\path[fill=lightgray] (5.5,2.2) rectangle (6.5,1.8);
\path[fill=lightgray] (4.3,0)--(4.3,-.55)--(5,-.55) arc[radius=.3,start angle=270, end angle=180]--(4.7,0);
\path[fill=lightgray] (7.2,0)--(7.2,-.55)--(6.5,-.55) arc[radius=.3,start angle=270,end angle=360]--(6.8,0);
\draw (4.3,-.55)--(4.3,3.05);
\draw (4.7,-.25)--(4.7,1.5);
\draw (7.2,-.55)--(7.2,3.05);
\draw (5,1.8) arc[radius=.3,start angle=90, end angle=180];
\draw (5,2.2) arc[radius=.3,start angle=270,end angle=90];
\draw (6.5,1.8) arc [radius=.3, start angle=90, end angle=0];
\draw (6.5,2.2) arc [radius=.3,start angle=270, end angle=450];
\draw (6.5,2.8)--(5,2.8);
\draw (6.8,1.5)--(6.8,-.25);
\draw (5,2.2)--(6,2.2);
\draw (5,1.8)--(6,1.8);
\draw (4.7,-.25) arc[radius=.3, start angle=180, end angle=270];
\draw (6.8,-.25) arc[radius=.3, start angle=0, end angle=-90];
\draw (5,-.55)--(6.5,-.55);
\draw [fill=white] (6,1.5) rectangle (6.5,2.5);
\draw [fill=white](4,0) rectangle (5,.5);
\draw [fill=white](6.5,0) rectangle (7.5,.5);
\draw [fill=white] (5,1.5) rectangle (5.5,2.5);
\node at (4,.25) [left] {$\$$};
\node at (6.5,.25)[left] {$\$$};
\node at (4.5,.25) {$e_P$};
\node at (7,.25) {$e_Q$};
\node at (5.25,1.5)[below] {$\$$};
\node at (6.25,1.5)[below] {$\$$};
\node at (5.25,2) {$e_P$};
\node at (6.25,2) {$e_Q$};
\node at (3,-2)[right] {$=\delta tr(e_Pe_Q)(e_P * e_Q)$};
\end{tikzpicture}
\end{equation*}

The above implies $\frac{1}{\delta tr(e_Pe_Q)}(e_P * e_Q)$ is a projection.
\end{proof}

\section{Angle and commuting square}
\label{Sec:angle}

Suppose $N\subseteq M$ is a finite index subfactor (not necessarily irreducible), $P$ is an intermediate subfactor and $e_P$ is the corresponding biprojection.
Denote $\tau_P=tr(e_P)$. Let $\delta= \sqrt{[M:N]}$ and $\tau={[M:N]}^{-1}$. Note: $tr(e_1)=\tau$.

\noindent \textbf{Notation}:
Define the unit vector $v_P$ in $(N'\cap M_1)_{s.a.}$ as
$$v_P= \frac{e_P-e_1}{{\lVert e_P-e_1\rVert}_2}.$$

For two intermediate subfactors $P,Q$ of a finite index subfactor $N\subseteq M$,
 we denote the quadruple of type $II_1$ factors $$\begin{matrix}
Q &\subset & M \cr
\cup &\ &\cup\cr
N  &\subset & P
\end{matrix},$$ by $(N,P,Q,M).$ 
We call it extremal, if the subfactor $N \subseteq M$ is extermal.

Recall the following definition (\cite{SW}):
\begin{definition}\label{Def:square}
 A quadruple $(N,P,Q,M)$ is called a commuting square if
$E^M_P E^M_Q= E^M_Q E^M_P = E^M_N$.
 A quadruple $(N,P,Q,M)$ is called a co-commuting square if the quadruple $(M,Q_1,P_1,M_1)$ is a commuting square.
 \end{definition}

\begin{definition}\label{alpha}
Suppose $P$ and $Q$ are intermediate subfactors of $N\subseteq M$.
Define the \textbf{angle}, denoted by $\alpha^N_M(P,Q)$, between $P$ and $Q$ as follows:
 \begin{equation*}
 \alpha^N_M(P, Q) =  \cos^{-1} {\langle v_P,v_Q\rangle},
 \end{equation*}
 where ${\langle x, y\rangle}= tr(y^*x)$ and 
 hence  ${\lVert x\rVert}= (tr(x^*x))^{1/2}$.
 \end{definition}
If $N$ and $M$ are clear from the context, we may omit them from ${\alpha}^N_M(P,Q).$
 
As usual, the angle takes only the principal value:
$$\displaystyle 0 \leq \alpha^N_M(P,Q) \leq \pi.$$
Note that $v_P, v_Q \geq 0$, so $\langle v_P,v_Q\rangle \geq 0$.
Therefore,
$$\displaystyle 0 \leq \alpha^N_M(P,Q) \leq \frac{\pi}{2}.$$

 \begin{proposition}
 For a quadruple $(N,P,Q,M)$, 
 $$\alpha(P,Q) = 0  \iff  e_P = e_Q.$$
\end{proposition}

\begin{proof}
Note that $\alpha(P,Q)=0$ iff $v_P$ is a multiple of $v_Q$. Since both $v_P$ and $v_Q$ are positive and $\lVert v_P\rVert= \lVert v_Q \rVert=1$, it follows that $v_P=v_Q.$  As $(e_P-e_1)$ and $(e_Q-e_1)$ are both projections, they are equal. So $e_P=e_Q$.
\end{proof}

\begin{proposition}\label{commuting}
The quadruple $(N,P,Q,M)$ forms a commuting square iff $\alpha(P,Q)= \frac{\pi}{2}$.
\end{proposition}
\begin{proof}
Note that
  $(N,P,Q,M)$ forms a commuting  square  iff $e_P e_Q= e_1$ iff $(e_P-e_1)(e_Q-e_1)=0$ iff $\alpha(P,Q) = \pi/2.$
\end{proof}

\begin{proposition}\label{Prop: correlation}
For an \textbf{extremal} quadruple $(N,P,Q,M)$,  
\begin{equation*}
\cos \alpha^N_M(P, Q) = corr(e_{P_1},e_{Q_1}),
\end{equation*}
where $corr(x,y):=\langle \frac{x-tr(x)}{\|x-tr(x)\|_2}, \frac{y-tr(y)}{\|y-tr(y)\|_2} \rangle $ is the correlation function. 
\end{proposition}

\begin{proof}
Let $\F: N' \cap M_1\to M' \cap M_2$ be the Fourier transform. The subfactor is extremal, so  
\begin{equation}\label{Equ:FPQ}
\cos \alpha^N_M(P, Q) =\langle v_P,v_Q\rangle=\langle \F(v_P),\F(v_Q)\rangle.
\end{equation}

Note that 
$\F(e_P)$ is a multiple of $e_{P_1}$ and $\F(e_1)$ is a multiple of the identity. So 
$$\F(e_P-e_1)=a e_{P_1}-b $$
for some constants $a$ and $b$.
Moreover,
$$\langle \F(e_P-e_1), \F(e_1)\rangle=\langle e_P-e_1, e_1\rangle=0,$$
so 
$$tr(\F(e_P-e_1))=0.$$ 
Therefore 
$$\F(e_P-e_1)=a(e_{P_1}-tr(e_{P_1})).$$ 
Recall that $\|v_P\|_2=1$, so 
$$\F(v_P)=\frac{e_{P_1}-tr(e_{P_1})}{\|e_{P_1}-tr(e_{P_1})\|_2}.$$
Similarly 
$$\F(v_Q)=\frac{e_{Q_1}-tr(e_{Q_1})}{\|e_{Q_1}-tr(e_{Q_1})\|_2}.$$
By Equation \eqref{Equ:FPQ},
\begin{equation*}
\cos \alpha^N_M(P, Q)= corr(e_{P_1},e_{Q_1}),
\end{equation*}

\end{proof}

\begin{definition}\label{beta}
We define the dual angle, denoted by $\beta^N_M(P,Q)$, between $P$ and $Q$ as 
$$\beta^N_M(P,Q):=\alpha^M_{M_1}(P_1, Q_1).$$
\end{definition}
This is similar to \cite{SW}.  As before, if from the context it is clear what $N$ and $M$ are, we may omit them from ${\beta}^N_M(P,Q).$
By duality, we have that 
\begin{align}\label{Equ:dual}
\alpha^N_M(P,Q)&=\beta^N_M(P_1,Q_1).
\end{align}
 
\begin{proposition}\label{co-commuting}
The quadruple $(N,P,Q,M)$ forms a co-commuting square iff $\beta(P,Q)= \frac{\pi}{2}$.
\end{proposition}
\begin{proof}
Follows from Definitions \ref{Def:square}, \ref{beta}, and Proposition \ref{commuting}. 
\end{proof} 
 
\begin{theorem}\label{Thm: alpha-beta}
For a quadruple $(N,P,Q,M)$, let $\tau_P=tr(e_P)$, $\tau_Q=tr(e_Q)$. Then
\begin{align}\label{Equ:alpha}
\cos \alpha^N_M(P, Q)  &=\frac{tr(e_Pe_Q)-\tau}{\sqrt{\tau_P-\tau}\sqrt{\tau_Q-\tau}}.
\end{align}
If the quadruple is extremal, then
\begin{align}\label{Equ:beta}
\cos \beta^N_M(P, Q) &= \frac{tr(e_Pe_Q)-\tau_P\tau_Q}{\sqrt{\tau_P-\tau_P^2}\sqrt{\tau_Q-\tau_Q^2}}.
\end{align}
\end{theorem}

\begin{proof}
By definition, 
\begin{align*}
\cos \alpha^N_M(P, Q) 
=&\langle v_P,v_Q\rangle \\
=& \frac{tr((e_P-e_1)(e_Q-e_1))}{{\lVert e_P-e_1\rVert}_2{\lVert e_Q-e_1\rVert}_2} \\
=&\frac{tr(e_Pe_Q)-\tau}{\sqrt{\tau_P-\tau}\sqrt{\tau_Q-\tau}}.
\end{align*}

By Proposition \ref{Prop: correlation} and duality,
\begin{align*}
\cos \beta^N_M(P, Q) 
=& \cos \alpha^M_{M_1}(P_1, Q_1)\\ 
=& corr(e_{P}, e_{Q})  \\
=&\frac{tr((e_P-\tau_P)(e_Q-\tau_Q))}{\|e_P-\tau_P\|_2\|e_Q-\tau_Q\|_2} \\
=& \frac{tr(e_Pe_Q)-\tau_P\tau_Q}{\sqrt{\tau_P-\tau_P^2}\sqrt{\tau_Q-\tau_Q^2}}.
\end{align*}

\end{proof}

\begin{corollary}\label{Cor:123}
Consequently 
\begin{align*}
 tr(e_Pe_Q) &\geq \tau, \\
 tr(e_Pe_Q) &\geq \tau_P\tau_Q.
\end{align*}
The equalities hold iff $\alpha^N_M(P,Q)=\pi/2$ and $\beta^N_M(P,Q)=\pi/2$ respectively.
\end{corollary}

\begin{corollary}\label{geq}
If the extremal quadruple $(N,P,Q,M)$ is a commuting square, then 
$[M:Q]\geq [P:N]$ and $[M:P] \geq [Q:N]$.
\end{corollary}

\begin{proof}
By Corollary \ref{Cor:123}, if the  quadruple $(N,P,Q,M)$ is a commuting square, then $\alpha^N_M(P,Q)=\pi/2$ and $tr(e_Pe_Q)=\tau$. So
$\tau \geq \tau_P\tau_Q$,
namely,
$$[M:N]\geq [P:N][Q:N].$$
Note that 
\begin{align*}
[M:N]&= [M:P][P:N],\\
[M:N]&= [M:Q][Q:N].
\end{align*} 
So 
\begin{align*}
[M:Q] \geq [P:N], \\
[M:P] \geq [Q:N].
\end{align*}
\end{proof}

\begin{remark}
This result was proved in Proposition 1.7 in \cite{Po5}, 
\end{remark}

\begin{definition}
For a quadruple $(N,P,Q,M)$, the following are equivalent:
\begin{enumerate}
\item $\tau_P \tau_Q=\tau$; 
\item $[M:P]=[Q:N]$;
\item $[M:Q]=[P:N]$. 
\end{enumerate}
We call the quadruple a parallelogram if one of the above equivalent conditions holds.
\end{definition}

In general, it is not true that $\alpha^N_M(P,Q)= \beta^N_M(P,Q)$( See for instance Fact \ref{inclusion}). One can have a quadruple which is commuting, but not co-commuting. 
Surprisingly, the following result holds:

\begin{theorem}\label{v}
If an extremal quadruple $(N,P,Q,M)$ is a parallelogram, then $\alpha^N_M(P,Q)=\beta^N_M(P,Q)$.
\end{theorem} 
Hence we may consider $\alpha^N_M(P,Q)$ and $\beta^N_M(P,Q)$ as opposite angles of the parallelogram.

\begin{proof}
If a quadruple $(N,P,Q,M)$ is a parallelogram, namely 
$$\tau= \tau_P\tau_Q,$$
then by Theorem \ref{Thm: alpha-beta}, 
$$\cos \alpha^N_M(P, Q)= \cos \beta^N_M(P, Q).$$
So 
$$\alpha^N_M(P, Q)= \beta^N_M(P, Q).$$
\end{proof}

 Motivated by \cite{SW} we try to  investigate the angle $\alpha^N_M(P,Q)$ in terms of Pimsner-Popa basis \cite{PiPo}.
 In this paper, by Pimsner-Popa basis we mean a `right basis'. Thus the condition for the set $\{\lambda_i:i\in I\}\subset M$ (for some finite index set $I$)
 to be a right basis for $M/N$ would be $\sum_{i=1}^n \lambda_i e_1 {\lambda_i}^*=1$ or equivalently,
 $x = \sum_{i=1}^n E_N(x \lambda_i){\lambda_i}^*
=\sum_{i=1}^n \lambda_i E_N({\lambda_i}^* x)$ for all $x\in M$. The set $\{\lambda_i:i\in I\}$ will be called a `left basis' for $M/N$ if $\{{\lambda_i}^*:i\in I\}$ is a right basis.

\begin{remark}
A set $\{\lambda_i:i\in I\}\subset M$ is call a two-sided basis for $M/N$, if it is both a left basis and a right basis. It is an open question whether any finite index, (irreducible) subfactor has a two-sided basis.
\end{remark}

\begin{proposition}\label{basis}
 Consider intermediate subfactors $P$ and $Q$ of $N\subset M$. Let $\{\lambda_i\}$(resp. $\{\mu_j\}$) be (right) basis
  for $P/N$ (resp $Q/N$). Then,
 \begin{equation}
  \cos(\alpha(P,Q))= \displaystyle \frac{\sum_{i,j} tr(E^M_N({\lambda_i}^*\mu_j) {\mu_j}^*\lambda_i)-1}{\sqrt{[P:N]-1}  \sqrt{[Q:N]- 1}}
 \end{equation}
 \end{proposition}
 \begin{proof}
  Firstly observe that for any intermediate subfactor, say $P$, of $N\subset M$ and basis $\{\lambda_i\}$ we have $e^M_P= \sum_i {\lambda_i}e_1 {\lambda_i}^*$. This follows trivally from the following array of equations and is well-known. For any $x\in M$, we have:
  \begin{align*}
   & (\sum_i \lambda_i e_1 {\lambda_i}^*)(x\Omega)\\
   & \qquad = (\sum_i {\lambda}_i (E^M_N({\lambda_i}^*x)))\Omega\\
   & \qquad = (\sum_i {\lambda}_i E^P_N({\lambda_i}^*E^M_P(x))) \Omega\\
   & \qquad = E^M_P(x) \Omega\\
   & \qquad = e^M_P(x\Omega).
   \end{align*}
In the above $\Omega$ denotes the cyclic vector for  the standard Hilbert space $L^2(M).$

In our notation, we have $e^M_Q= \sum_j \mu_j e_1 {{\mu}_j}^*$. Then it follows from Definition \ref{alpha} that:
\begin{align*}
 & \cos(\alpha(P,Q))\\
 & \qquad= \displaystyle \frac{tr(e_Pe_Q)- \tau}{\sqrt{tr(e_P)-\tau}\sqrt{tr(e_Q)-\tau}}\\
 & \qquad = \displaystyle \frac{tr(\sum_{i,j} {\lambda}_i e_1 {\lambda_i}^* \mu_j e_1 {\mu_j}^*)-\tau}{\sqrt{tr(\sum_i {\lambda}_i e_1 {{\lambda}_i}^*)- \tau}\sqrt{tr(\sum_j\mu_j e_1 {\mu_j}^*)-\tau}}\\
 & \qquad = \displaystyle \frac{\sum_{i,j} tr(e_1 E^M_N({\lambda_i}^*\mu_j) {\mu_j}^*\lambda_i)-\tau}{\sqrt{\sum_i tr(e_1  {\lambda_i}^*\lambda_i)- \tau} \sqrt{\sum_j tr(e_1  {\mu_j}^*\mu_j)- \tau}}\\
 & \qquad = \displaystyle \frac{\sum_{i,j} tr(E^M_N({\lambda_i}^*\mu_j) {\mu_j}^*\lambda_i)-1}{\sqrt{[P:N]-1}  \sqrt{[Q:N]- 1}}
 \end{align*}
This completes the proof.
 \end{proof}

\begin{fact}\label{inclusion}
 Consider intermediate subfactors $P$ and $Q$ such that $N\subset P\subset Q \subset M$. Then the following two
 two equations hold (as is seen from the definitions):
 
 $$\cos(\alpha(P,Q))= \sqrt{\frac{[P:N]-1}{[Q:N]-1}}.$$ and 
 $$\cos(\beta(P,Q))= \sqrt{\frac{[M:Q]-1}{[M:P]-1}}.$$
 
 This shows that $\alpha(P,Q)$ and $\beta(P,Q)$ may not be equal. For example,
 consider subfactors  $N\subset P\subset Q\subset M$ such that 
 $[P:N]=2,[M:Q]= 3,[Q:P]=5$. Then by the above two formulas we get $\cos(\alpha(P,Q))=\frac{1}{3}$ and
 $\cos(\beta(P,Q))= \frac{1}{\sqrt{7}}$.
\end{fact}
\begin{proposition}
  Consider factors of type $II_1$ such that $R, N\subset P,Q\subset M, S$. Then $\alpha^N_M(P,Q)= \alpha^N_S(P,Q)$
  and  $\beta^N_M(P,Q)= \beta^R_M(P,Q)$.
\end{proposition}
\begin{proof}
 This follows from Proposition \ref{basis}.
\end{proof}
\begin{definition}\label{p and q}
 Consider the quadruple of type $II_1$ factors $(N,P,Q,M).$ Let $\{\lambda_i\}$(resp. $\{\mu_j\}$) be a basis for $P/N$ (resp. $Q/N$).
 Define two self-adjoint operators $p$ and $q$ as follows:
 $$p:= \sum_{i,j}{\lambda_i}\mu_j e_1 {\mu_j}^*{\lambda_i}^*~ and ~q:= \sum_{i,j}\mu_j \lambda_i e_1 {\lambda_i}^* {\mu_j}^*.$$
In general, $p$ and $q$ are not projections. Later we will see that they always have same spectrum and have the same trace.
\end{definition}

\begin{lemma}\label{Lem:p=1}
Following the notations in Definition \ref{p and q}, $\{\lambda_i\mu_j\}$ is a basis for $M/N$ iff $p=1$, and $\{\mu_j\lambda_i\}$ is a basis for $M/N$ if and only if $q=1$.
\end{lemma}

\begin{proof}
Follows from the definition of Pimsner-Popa basis.
\end{proof}

\begin{lemma}\label{independent}
The definition above (of $p$ and $q$) does not depend on the basis we have chosen.
 \end{lemma}
 \begin{proof}
  Suppose, $\{\psi_j:j\in I\}$ is another basis for $P/N$. Then it is easy to see that:
  
  \begin{align*}
   \sum_i \lambda_i e_Q {\lambda_i}^*= & \sum_{i} \{\sum_j \psi_j E^P_N({\psi_j}^*\lambda_i)\}e_Q{\lambda_i}^*\\
   &= \sum_{i,j}\psi_je_QE^P_N({\psi_j}^*\lambda_i){\lambda_i}^*\\
   & = \sum_j \psi_je_Q\{\sum_i  E^P_N({\psi_j}^*\lambda_i){\lambda_i}^*\}\\
   &= \sum_j \psi_j e_Q{\psi_j}^*.
  \end{align*}
  As already observed in the proof of Proposition \ref{basis},
 $e_Q= \sum_j \mu_j e_1 {{\mu}_j}^*$. Thus $p= \sum_i \lambda_i e_Q {\lambda_i}^*.$
This shows that $p$ is independent of basis chosen. Similar proof works for $q$.
 \end{proof}

\begin{proposition}\label{q}

 Consider again $N\subset P,Q\subset M$ and let $\{\lambda_i\}$(resp. $\{\mu_j\}$) be a basis for $P/N$ (resp. $Q/N$). Then the following are equivalent:
 \begin{enumerate}
 \item $\alpha(P,Q)= \pi/2$
 
 \item  $q:= \sum_{i,j}\mu_j \lambda_i e_1 {\lambda_i}^* {\mu_j}^*$ is a projection such that $q \geq e_P$.
  
 \item $p:= \sum_{i,j}{\lambda_i}\mu_j e_1 {\mu_j}^*{\lambda_i}^* $ is a projection such that $p\geq e_Q.$
\end{enumerate}
 \end{proposition}

\begin{proof}
$(1)\Rightarrow (2)$

That $q$ is a projection is easy and was observed in \cite{SW}. We prove it for sake of completeness.
 By the second line of the proof of Proposition \ref{basis}, $q= \sum_i {\mu}_i e_P {\mu_i}^*$ and hence $q=q^*$.

 Then,
 \begin{align*}
  q^2 = &  \sum_{i,j} \mu_i e_P {\mu_i}^*\mu_j e_P {\mu_j}^*\\
  & = \sum_{i,j} {\mu_i} E^M_P({\mu_i}^*\mu_j) e_P {\mu_j}^*\\
  & = \sum_{i,j} \mu_i E^M_P E^M_Q({\mu_i}^*\mu_j)e_P {\mu_j}^*\\
  & = \sum_{i,j} \mu_i E^M_N({\mu_i}^*\mu_j)e_P {\mu_j}^*~[\textrm{applying Proposition}~ \ref{commuting}]\\
  &= \sum_j \mu_j e_P {\mu_j}^*~ [\textrm{since} \{\mu_j\} \textrm{is a basis for}~ Q/N]\\
  & = q.
 \end{align*}
Now we show that $(e_P) q= e_P$.
\begin{align*}
 (e_P) q= & \sum_je_P \mu_j e_P {\mu_j}^*\\
 &= \sum_je_P E^M_P(\mu_j){\mu_j}^*\\
 & = \sum_j e_P E^M_P E^M_Q(\mu_j) {\mu_j}^*\\
 & = \sum_j e_P E^M_N(\mu_j) {\mu_j}^*~ [\textrm{applying Proposition} ~\ref{commuting}]\\
 & = e_P ~[\textrm{since} ~\{\mu_j\} ~\textrm{is a basis for}~ Q/N]
\end{align*}
Thus $q$ is projection such that $q\geq e_P$. 
This completes the proof of $(1)\Rightarrow (2)$.
\bigskip

$(2)\Rightarrow (1)$

$(e_P)q= e_P$ implies $\sum_j e_P E^M_P(\mu_j) {\mu_j}^*= e_P$. Taking trace to both sides we get, 
\begin{equation}\label{eq}
\sum_j tr( E^M_P(\mu_j){\mu_j}^*)=1.
\end{equation}

 Then from the definition of angle it follows easily that
\begin{equation}\label{de}
\cos(\alpha(P,Q))= \displaybreak \frac{tr(e_Pe_Q)-\tau}{\sqrt{tr(e_P)-\tau}\sqrt{tr(e_Q)-\tau}}
\end{equation}
Put $\displaystyle r= \sum_j {\mu_j}^* e_P \mu_j$. Thus, 
\begin{eqnarray*}
   tr(re_1)& = & tr(\sum_j{\mu_j}^*e_P\mu_je_1)\\
   & = & tr(e_P\sum_j\mu_je_1{\mu_j}^*)\\
   & = &tr(e_P e_Q)~~[\textrm{since}~\sum_j\mu_je_1{\mu_j}^*=e_Q].
\end{eqnarray*}
Thus it follows from Equation \ref{de} that:
\begin{equation}\label{c}
 \cos(\alpha(P,Q))= \displaybreak \frac{tr(re_1)-\tau}{\sqrt{tr(e_P)-\tau} \sqrt{tr(e_Q)- \tau}}
\end{equation}

On the other hand,
\begin{align*}
 re_1= & \sum_j {\mu_j}^* e_P \mu_j e_1\\
 & = \sum_j {\mu_j}^* e_P \mu_j e_P e_1~~~[\textrm{since}~~e_Pe_1= e_1]\\
 &= \sum_j {\mu_j}^* E^M_P(\mu_j) e_1
\end{align*}

Thus $tr(re_1)= \tau tr({\mu_j}^* E^M_P(\mu_j))= \tau tr(E^M_P(\mu_j) {\mu_j}^*)$. Then Equation (\ref{eq}) implies that $tr(re_1)= \tau.$
Thus by Equation(\ref{c}) $\alpha(P,Q)= \pi/2.$

This completes the proof of $(2)\Rightarrow (1)$.
\bigskip

$(1)\Leftrightarrow (3)$
\smallskip

Simply observe that $\alpha(P,Q)= \alpha(Q,P)$. The rest follows from above two implications.
This completes the proof.
\end{proof}
\bigskip

\begin{fact}
   $q=e_P$ if and only if $Q=N$. Similarly $p=e_Q$ if and only if $P=N$.
 \begin{proof}
 Firstly observe, by Markov property of trace,
 $$tr(q)= tr(\sum_j \mu_j e_P {\mu_j}^*)= \frac{\sum_j tr(\mu_j {\mu_j}^*)}{[M:P]}.$$
But as $\{\mu_j\}$ is a basis for $Q/N,~~ \sum_j\mu_j {\mu_j}^*= [Q:N]$. Thus
\begin{equation}\label{trq}
tr(q)= \frac{[M:N]}{[M:P][M:Q]}.
\end{equation}
Suppose $q= e_P$. After taking trace  on both sides we get $ [M:N]= [M:Q]$ implying $Q=N$.

Conversely $Q=N$ implies $tr(q)= tr(e_P)$ (See Equation (\ref{trq})). Since by Proposition \ref{q} $q\geq e_P$, it follows that $q=e_P$, as $tr$ is faithful.
\end{proof}

\end{fact}

\begin{proposition}\label{J}
  Consider again $N\subset P,Q\subset M$ and let $\{\lambda_i\}$(resp. $\{\mu_j\}$) be a basis for $P/N$ (resp. $Q/N$). Define $p$ and $q$ as in Proposition \ref{q}. Then $JpJ=q,$ where $J$ is the ususal modular conjugation operator on $L^2(M).$ 
\end{proposition}

 \begin{proof}
 Firstly we know $p=\sum_i{\lambda}_ie_Q{{\lambda}_i}^*$ and $q=\sum_j{\mu}_je_P{{\mu}_j}^*.$ Let us denote by $\Omega$ the cyclic vector for  the standard Hilbert space $L^2(M).$ Then for any $x\in M$,
 \begin{align*}
  & JpJ(x\Omega)= Jp(x^*\Omega)\\
  & \qquad \qquad  = J(\sum_i{\lambda}_ie_Q({{\lambda}_i}^*x^*\Omega))\\
  & \qquad \qquad =\sum_i J({\lambda}_iE^M_Q({\lambda_i}^*x^*)\Omega)\\
  & \qquad \qquad = \sum_i(E^M_Q(x{\lambda}_i){{\lambda}_i}^*)\Omega\\
  & \qquad \qquad = \sum_i(\sum_j\mu_jE^Q_N\{{\mu_j}^*E^M_Q(x\lambda_i)\}{\lambda_i}^*)\Omega~~~~~~~~~~~~\textrm{[since}~\{\mu_j\}~~\textrm{is a basis for}~ Q/N]\\
  & \qquad \qquad = \sum_{i,j}(\mu_j E^Q_N\{E^M_Q({\mu_j}^*x{\lambda}_i)\}{\lambda_i}^*)\Omega\\
  &\qquad \qquad = \sum_{i,j}(\mu_jE^M_N({\mu_j}^*x{\lambda}_i){\lambda_i}^*)\Omega.
 \end{align*}
On the other hand the following array of equations hold true;
\begin{align*}
& q(x\Omega) = (\sum_j{\mu}_je_P{{\mu}_j}^*)(x\Omega)\\
&\qquad \qquad = \sum_j (\mu_j E^M_P({\mu_j}^*x))\Omega\\
& \qquad \qquad = \sum_j (\mu_j(\sum_i E^P_N\{E^M_P({\mu_j}^*x)\lambda_i\}{\lambda_i}^*))\Omega ~~~~~~~~~~~\textrm{[since}~\{\lambda_i\}~~\textrm{is a basis for}~ P/N]\\
& \qquad \qquad = \sum_{i,j}(\mu_j E^P_N(E^M_P({\mu_j}^*x\lambda_i)){\lambda_i}^*)\Omega\\
  &\qquad \qquad = \sum_{i,j}(\mu_jE^M_N({\mu_j}^*x{\lambda}_i){\lambda_i}^*)\Omega.
\end{align*}

Thus we see that $JpJ=q$. This completes the proof.
\end{proof}

The following result is well-known, for example see Proposition 2.7 in \cite{Bi2}:\begin{lemma}\label{ocneanu}
 Let $N\subset M$ be an inclusion of $II_1$ factors with finite index, and let $\{m_i:i\in I\} \subset M$ be a Pimsner-Popa basis(not necessarily orthonormal)
 for $M/N$. Let us also  denote by $tr_{N^{\prime}}$  the unique normalized trace on $N^{\prime}=N^{\prime}\cap \mathcal{B}(L^2(M)).$ Then the unique $tr_{N^{\prime}}$-preserving conditional expectation is given by the following map $\phi$:
 $$\phi(x)={[M:N]}^{-1}\sum_im_ix{m_i}^*.$$
where $x\in N^{\prime}.$
\end{lemma}

\begin{proposition}
 Let $N\subset P,Q\subset M$ be intermediate subfactors such that $[M:N]$ is finite, not necessarily irreducible. Then the  self-adjoint operator $p$ belongs to  $P^{\prime}\cap Q_1$ and is given by $p=[P:N] E^{N^{\prime}}_{P^{\prime}}(e_Q)= [Q:N] E^{M_1}_{Q_1}(e_P).$ Similarly, $q= [Q:N]E^{N^{\prime}}_{Q^{\prime}}(e_P)= [P:N] E^{M_1}_{P_1}(e_Q)$ also. Thus $q\in Q^{\prime}\cap Q_1$.
\end{proposition}

\begin{proof} 
Consider again $N\subset P,Q\subset M$ and let $\{\lambda_i\}$(resp. $\{\mu_j\}$) be a basis for $P/N$ (resp. $Q/N$)
 By Lemma \ref{ocneanu} we immediately get, for any $x\in N^{\prime}$, 
 $$E^{N^{\prime}}_{P^{\prime}}(x)={[P:N]}^{-1}\sum_i {\lambda}_ix{\lambda_i}^*.$$
 Clearly $e_Q\in N^{\prime}$ and hence
 $E^{N^{\prime}}_{P^{\prime}}(e_Q)={[P:N]}^{-1}\sum_i {\lambda}_ie_Q{\lambda_i}^*$. Thus, $p= \sum_i {\lambda}_ie_Q{\lambda_i}^*= [P:N] E^{N^{\prime}}_{P^{\prime}}(e_Q).$ Similarly, we can prove
 $q=[Q:N] E^{N^{\prime}}_{Q^{\prime}}(e_P)$. Now take the modular conjugation operator $J$ on $L^2(M)$ to get $JqJ= [Q:N]  E^{M_1}_{Q_1}(e_P).$
 Now, by Proposition \ref{J} we immediately get $p= [Q:N]  E^{M_1}_{Q_1}(e_P).$ Proof for $q$ is similar. This completes the proof of the proposition.
\end{proof}

\begin{proposition}\label{central}
 Let $\alpha= \pi/2$ and $p,q$ be as in Theorem \ref{q}. Then,
 $$\bigvee\{ve_Qv^*: v\in \mathcal{U}(P)\}=p$$ and
 $$ \bigvee\{ue_Pu^*: u\in \mathcal{U}(Q)\}=q$$
\end{proposition}

\begin{proof}
 First note that as observed in Proposition \ref{q} for any  basis $\{\mu_j\}$ of $Q/N$, $q= \sum_j \mu_j e_P {\mu_j}^*$ is a projection such that $q\geq e_P$.
 Consider an arbitrary unitary element $u\in \mathcal{U}(Q)$. Then it is trivial to see that $\{u^*\mu_j\}$ is a basis for $Q/N$.
 Thus $u^*qu \geq e_P$ and hence $ue_Pu^*\leq q$. Therefore, $\bigvee \{ue_Pu^*:u\in \mathcal{U}(Q)\}\leq q$.
  Observe that, since $q=\sum_j \mu_j e_P {\mu_j}^*$,
  \begin{eqnarray*}
   range(q)& \subset &[\mu L^2(P):\mu \in Q]\\
   & = & [u L^2(P): u\in \mathcal{U}(Q)]\\
   & = &[range(\{ue_Pu^*: u\in \mathcal{U}(Q)\})].
   \end{eqnarray*}
 Thus, $\bigvee\{ue_Pu^*:u\in \mathcal{U}(Q)\}\geq q$.
So, $\bigvee \{ue_P u^*:u\in \mathcal{U}(Q)\}= q$. Proof for $p$ is exactly similar.
\end{proof}

\begin{remark}
Let $\alpha= \pi/2$ and $p,q$ be as in Theorem \ref{q}.
Then it is not hard to show that $p,q \geq e_P \vee e_Q.$ In general, it is not true that $e_P \vee e_Q= e_{P\vee Q}$,
 although $e_P\vee e_Q\leq p,q\leq e_{P\vee Q}.$
\end{remark}

Below we give a characterization of commuting square in terms of basis:
\begin{theorem}\label{cs}
For a quadruple $(N, P,Q,M)$ the following are equivalent:
\begin{enumerate}

\item
$(N,P,Q,M)$ is a commuting square, that is $\alpha(P,Q) = \pi/2$.

\item $$p= \bigvee \{v e_Qv^*: v\in \mathcal{U}(P)\}.$$

\item  $$q=\bigvee \{u e_Pu^*: u\in \mathcal{U}(Q)\}.$$
\end{enumerate}
\end{theorem}
\begin{proof}
 $(1) \Rightarrow (2)$ 

This is Proposition \ref{central}.

\bigskip

$(2) \Rightarrow (1)$

Clearly $\bigvee \{v e_Qv^*: v\in \mathcal{U}(P)\} \geq e_Q$. Hence $p\geq e_Q$. Again applying Proposition \ref{q}
we get $\alpha(P,Q)= \pi/2$.

Thus (1) and (2) are equivalent.

By symmetry,  (1) and (3) are equivalent.
This completes the proof.
\end{proof}

Below we investigate when $\alpha(P,Q)= \pi/2= \beta(P,Q)$. Explicitly we characterize  simultaneously commuting and co-commuting squares
in terms of various equivalent conditions. 

\begin{theorem}\label{sym}
 For a quadruple $(N,P,Q,M)$, the following are equivalent:
 \begin{enumerate}
  \item $(N,P,Q,M)$ is a commuting and co-commuting square; 
  \item  $\alpha(P,Q)= \beta(P,Q)= \pi/2$;
  \item $p=1$;
  \item  If $\{\lambda_i\}$(resp. $\{\mu_j\})$ is a basis for $P/N$ (resp. $Q/N)$, then $ \{\lambda_i \mu_j\}$ is a basis for $M/N$;
  \item $q=1$;
  \item  If $\{\lambda_i\}$(resp. $\{\mu_j\})$ is a basis for $P/N$ (resp. $Q/N)$,then  $ \{\mu_j \lambda_i\}$ is a basis for $M/N$;
  
  \item  Any basis(not necessarily orthonormal) for $P/N$ is a basis for $M/Q$;
  \item  Any basis (not necessarily orthonormal) for $Q/N$ is a basis for $M/P$.
         
 \end{enumerate}
 \end{theorem}
 
 \begin{proof}
  By Proposition \ref{commuting} and by definition of co-commuting square $(1) \iff (2)$.
  
 Suppose (1) holds true. Then, applying Corollary \ref{geq} twice we get $[M:Q]=[P:N]$. Thus by Equation (\ref{trq}) $tr(q)=1$. 
 Since $\alpha(P,Q)=\pi/2$ from Proposition \ref{q} it follows that $q$ is  a projection implying $q=1$. Similarly $p=1$.
 Thus $(1)\Rightarrow (3),(5).$

By Lemma \ref{Lem:p=1}, $(3)\iff (4)$ and $(5)\iff (6)$.

Suppose $(3)$ holds true, that is  $p=1$ . Thus applying Proposition \ref{q} we immediately get $\alpha(P,Q)=\pi/2$. 
Using again Equation(\ref{trq}) we obtain $[M:Q]=[P:N]$.
 Then Theorem \ref{v} implies $\beta(P,Q) = \pi/2$. In other words, $(3)\Rightarrow (1)$.

 Suppose (4) holds true. Let $\{\lambda_i\}$ be any basis for $P/N$. Fix a basis $\{\mu_j\}$ for $Q/N$.
 Thus, $(4)$ implies $\{\lambda_i\mu_j\}$ is a basis for $M/N$.
 Hence, $\sum_{i,j} \lambda_i \mu_j e_1 {\mu_j}^*{\lambda_i}^*=1$. Thus, $\sum_i \lambda_i e_Q {\lambda_i}^*=1$ (since we know $\sum_j \mu_j e_1{\mu_j}^*=e_Q$). We obtain $\{\lambda_i\}$ is a basis for $M/Q$.
 Therefore, we obtain $(4) \Rightarrow (7).$

 Simply use Lemma 4.3.4 (i) of \cite{JS} to conclude that $(7) \Rightarrow (4)$.

 Therefore we obtain, Therefore $(1) \iff (2) \iff (3) \iff (4)\iff (7)$.
 
 By symmetry (that is $\beta(P,Q)= \beta(Q,P)) (1) \iff (2) \iff (5) \iff (6)\iff (8)$.
 
 This completes the proof of equivalent statements.
 \end{proof}

Now, the following corollary follows easily. This is the characterization of non-degenerate commuting square due to S. Popa (see \cite{Po2})(with slight modification):

\begin{corollary}\cite{Po2}\label{po2}
For a commuting square $(N,P,Q,M)$ of $II_1$-factors with all inclusions of finite index, the following statements are equivalent:
 \begin{enumerate}
 
 \item
  $(N,P,Q,M)$ is a co-commuting square, that is $\beta^N_M(P,Q)= \pi/2.$

   \item $\bigvee\{ve_Qv^*: v\in \mathcal{U}(P)\}=1$

 \item  $\bigvee\{ue_Pu^*: u\in \mathcal{U}(Q)\}=1$

 \item  Any basis(not necessarily orthonormal) for $P/N$ is a basis for $M/Q$.

 \item Any  basis(not necessarily orthonormal) for $Q/N$ is a basis for $M/P$.

 \item  $PQ:= span\{\sum_{i=1}^n x_iy_i:x_i\in P, y_i\in Q\} = M$, in particular $(N,P,Q,M)$ is non-degenerate.

 \item  $QP= M$, in particular $(N,Q,P,M)$ is non-degenerate
 \end{enumerate}
\end{corollary}
\begin{proof}
Suppose $\{\lambda_i\}, \{\mu_j\},p$ and $q$ be as before.

By Theorem \ref{sym} and Proposition \ref{central} it is trivial to see that conditions (1), (2) and (4) all are 
equivalent to satisfy the equation  $p=1$. Similarly, (1),(3) and (5) are equivalent to the equation $q=1$.

Suppose $(3)$ holds true. Thus by Theorem \ref{cs} $\{\mu_j \lambda_i\}$ is a basis for $M/N$ and hence $M=QP$ implying (7). Conversely, suppose $(7)$ holds true. Thus  any $x\in M$ can be written as $x= \sum_k b_k a_k$, where $b_k\in Q$ and $a_k\in P$.
Then  it is easy to check that for any $x\in M$:
\begin{align*}
 q(x\Omega) = & q((\sum_k b_k a_k)\Omega)\\
 & \qquad = \sum_{j,k} \mu_j e_P({\mu_j}^*b_k a_k\Omega)\\
 & \qquad= \sum_{j,k} \mu_j E^M_P({\mu_j}^*b_k)a_k\Omega\\
 & \qquad= \sum_{j,k} \mu_j E^M_P E^M_Q({\mu_j}^*b_k)a_k\Omega\\
 & \qquad= \sum_{j,k} \mu_j E^Q_N({\mu_j}^*b_k)a_k\Omega~ ~[\textrm{by commuting square condition}]\\
 & \qquad= \sum_{k} b_k a_k \Omega~~ [\textrm{since}~\{\mu_j\} ~\textrm{is a basis for}~ Q/N]\\
 &\qquad= x\Omega.
\end{align*}
Thus $q=1$.  

That $(6)$ is equivalent to $ p=1$ is exactly similar.

This completes the proof.
\end{proof}

\begin{remark}
 It is worth mentioning that Popa has shown  that if (4) of Theorem \ref{sym} holds for a quadruple $(N,P,Q,M)$, then $\overline{sp PQ}=M$ with the additional 
 assumption that the quadruple is a commuting square; whereas we have shown in Theorem \ref{sym} that if (4) holds, then automatically the quadruple will be a non-degenerate commuting square.
\end{remark}

\begin{corollary}
 Let $(N,P,Q,M)$ be a quadruple. If for some basis $\{\mu_j\}$  for $Q/N$ it happens that  $\{{\mu_j}^*\}$
 is  a basis for $M/P$, then $\alpha^N_M(P,Q)= \beta^N_M(P,Q)= \pi/2$. Similar statement holds for $\{\lambda_i\}$.
\end{corollary}
\begin{proof}
 Put as before $r= \sum_j{\mu_j}^* e_P \mu_j$. By assumption $r=1$. Thus by Equation (\ref{c}) $\alpha(P,Q)= \pi/2$. By property of 
 basis $\sum_j \mu_j {\mu_j}^*= [Q:N]$ and $\sum_j {\mu_j}^* \mu_j= [M:P]$. Thus $[M:P]=[Q:N]$ and hence the quadruple is a parallelogram and therefore by Theorem \ref{v} $\beta(P,Q)= \pi/2$.
\end{proof}

\begin{corollary}\label{twosided}
 If $P/N$ and $Q/N$ both have two sided basis, then $\alpha(P,Q)= \beta(P,Q)= \pi/2$ implies that $M/N$ has two sided basis.
\end{corollary}
\begin{proof}
 Just use the fact $(2)\Leftrightarrow (3)$ of Theorem \ref{sym}.
\end{proof}
\begin{corollary}
Consider  the intermediate subfactor $P$ such that $N\subset P \subset M$.
 Suppose, $\{\lambda_i\}$(respectively, $\{\gamma_j\}$) is a two-sided basis for $P/N$(resp. for $M/P$).
 If there exists another intermediate subfactor $Q$ such that $\alpha(P,Q)= \beta(P,Q)= \pi/2$, then
 $\{\lambda_i \gamma_j\}$ is a two-sided basis for $M/N$.
\end{corollary}

\begin{proof}
 Firstly applying Theorem \ref{po2} we find that $\{\gamma_j\}$ is a two sided basis for $Q/N$. Then simply using Corollary \ref{twosided} we immediately obtain the result.
\end{proof}

\section{Boundedness  of angle}\label{Sec:rigidity}

In this section  and the next section we assume that $N\subseteq M$ is  an irreducible subfactor. 
In the irreducible case intermediate von-Neumann algebras are intermediate subfactors, so the set of intermediate subfactors form a lattice under the operations $P\wedge Q= P\cap Q$ and $P\vee Q= \{P\cup Q\}^{\dprime}$.

\begin{definition}
 Let $N\subseteq M$ be a subfactor. Then $Q$ is called a maximal (respectively, minimal) intermediate subfactor of $N\subseteq M$ if whenever there exists an intermediate subfactor $P$ such that $N\subseteq Q\subseteq P\subseteq M$ (respectively, $N\subseteq P\subseteq Q\subseteq M$), then
 $P$ equals either $Q$ or $M$(respectively, $P$ equals either $N$ or $Q$). We 
 exclude $N$ and $M$ from the definition of maximal(or minimal) intermediate subfactor for obvious reason. \par Note that, maximal intermediate subfactors in
$N\subseteq M$ correspond to minimal intermediate subfactors in $M\subseteq M_1$.
\end{definition}
\begin{lemma}\label{lem:epeq}
	Suppose $e_P$ and $e_Q$ are two biprojections, then $e_P\vee e_Q$ is a subprojection of $\frac{1}{\delta tr(e_Pe_Q)}e_P\star e_Q$.
\end{lemma}
\begin{proof}
Using exchange relation for the biprojection $e_P$ we get the equations as in Fig.~\ref{pe1}:
 \begin{figure}[H]\centering
 	\begin{tikzpicture}
 	\path [fill=lightgray] (1.3,3)--(1.7,3)--(1.7,1.6) arc [radius=.3,start angle=180, end angle=270]--(2.4,1.3) arc[radius=.3, start angle=90, end angle=0]--(2.7,-.5)--(2,-.5) arc[radius=.3, start angle=-90, end angle=0]--(2.3,-.2)--(2.3,.7) arc[radius=.3, start angle=0, end angle=90]--(1,1) arc[radius=.3, start angle=90, end angle=180]--(.7,-.2) arc [radius=.3, start angle=180, end angle=270]--(.3,-.5)--(.3,1) arc[radius=.3,start angle=180, end angle=90]--(1,1.3) arc[radius=.3, start angle=270, end angle=360];
 	\draw (.3,-.5)--(.3,1);
 	\draw (2.7,-.5)--(2.7,1);
 	\draw (.7,-.2)--(.7,.7);
 	\draw (2.3,-.2)--(2.3,.7);
 	\draw (.7,-.2) arc[radius=.3, start angle=180, end angle=270];
 	\draw (2.3,-.2) arc[radius=.3,start angle=0, end angle=-90];
 	\draw (.7,.7) arc[radius=.3, start angle=180, end angle=90];
 	\draw (2.3,.7) arc[radius=.3,start angle=0, end angle=90];
 	\draw (2,1)--(1,1);
 	\draw (.3,1) arc [radius=.3, start angle=180, end angle=90];
 	\draw (.6,1.3)--(1,1.3);
 	\draw (1.3,1.6)--(1.3,3);
 	\draw (1.3,1.6) arc [radius=.3, start angle=0, end angle=-90];
 	\draw (1.7,1.6)--(1.7,3);
 	\draw (1.7,1.6) arc [radius=.3, start angle=180, end angle=270];
 	\draw (2.7,1) arc [radius=.3, start angle=0, end angle=90];
 	\draw (2.4,1.3)--(2,1.3);
 	\draw (1,-.5)--(2,-.5);
 	\draw [fill=white] (0,0) rectangle (1,.5);
 	\draw [fill=white] (2,0) rectangle (3,.5); 
 	\draw [fill=white] (1,2) rectangle (2,2.5);
 	\node at (0,.25) [left] {$\$$};
 	\node at (2,.25) [left] {$\$$};	
 	\node at (1,2.25)[left] {$\$$};
 	\node at (.5,.25){$e_P$};
 	\node at (2.5,.25){$e_Q$};
 	\node at (1.5,2.25){$e_P$};
 	\node at (3.5,1.25)[right] {$=$};
 	\path [fill=lightgray] (-1+6.3,3)--(-1+6.3,-.5)--(-1+7.7,-.5)--(-1+7.7,1)--(-1+7.3,1)--(-1+7.3,-.1) arc [radius=.2,start angle=0, end angle=-90]--(-1+6.9,-.3) arc[radius=.2, start angle=270, end angle=180]--(-1+6.7,1.2) arc [radius=.2,start angle=180, end angle=90]--(-1+7.1,1.4) arc [radius=.2,start angle=90, end angle=0]--(-1+7.3,1.2)--(-1+7.7,1.2)--(-1+7.7,1.4) arc[radius=.2,start angle=0,end angle=90]--(-1+6.9,1.6) arc[radius=.2,start angle=270, end angle=180]--(-1+6.7,3);
 	\draw (-1+6.3,3)--(-1+6.3,-.5);
 	\draw (-1+6.7,3)--(-1+6.7,1.8);
 	\draw (-1+6.7,1.8) arc[radius=.2,start angle=180, end angle=270];
 	\draw (-1+6.9,1.6)--(-1+7.5,1.6);
 	\draw (-1+7.5,1.6) arc [radius=.2,start angle=90, end angle=0];
 	\draw (-1+7.7,1.4)--(-1+7.7,-.5);
 	\draw (-1+7.3,1.2) arc[radius=.2,start angle=0, end angle=90];
 	\draw (-1+7.1,1.4)--(-1+6.9,1.4);
 	\draw (-1+6.9,1.4) arc [radius=.2, start angle=90, end angle=180];
 	\draw (-1+6.7,1.2)--(-1+6.7,-.1);
 	\draw (-1+6.7,-.1) arc[radius=.2, start angle=180,end angle=270];
 	\draw (-1+6.9,-.3)--(-1+7.1,-.3);
 	\draw (-1+7.1,-.3) arc [radius=.2, start angle=270, end angle=360];
 	\draw (-1+7.3,-.1)--(-1+7.3,1.2);
 	\draw [fill=white] (-1+6,2) rectangle (-1+7,2.5);
 	\draw [fill=white] (-1+7,0) rectangle (-1+8,.5);
 	\draw [fill=white] (-1+7,.7) rectangle (-1+8,1.2);
 	\node at (-1+7.1,.25)[left]{$\$$};
 	\node at (-1+7.1,.95)[left]{$\$$};
 	\node at (-1+6.1,2.25)[left]{$\$$};
 	\node at (-1+6.5,2.25){$e_P$};
 	\node at (-1+7.5,.25){$e_Q$};
 	\node at (-1+7.5,.95){$e_P$};
 	\node at (8,1.25)[right]{$=\delta tr(e_Pe_Q)e_P$};
 	\end{tikzpicture}\caption{$e_P$ is a subprojection of $\frac{1}{\delta tr(e_Pe_Q)}e_P\star e_Q$}
	\label{pe1}
 \end{figure}

Using exchange relation for the biprojection $e_Q$ we get the equations as in Fig.~\ref{qe1}:
 \begin{figure}[H]\centering
	\begin{tikzpicture}
	\path [fill=lightgray] (1.3,3)--(1.7,3)--(1.7,1.6) arc [radius=.3,start angle=180, end angle=270]--(2.4,1.3) arc[radius=.3, start angle=90, end angle=0]--(2.7,-.5)--(2,-.5) arc[radius=.3, start angle=-90, end angle=0]--(2.3,-.2)--(2.3,.7) arc[radius=.3, start angle=0, end angle=90]--(1,1) arc[radius=.3, start angle=90, end angle=180]--(.7,-.2) arc [radius=.3, start angle=180, end angle=270]--(.3,-.5)--(.3,1) arc[radius=.3,start angle=180, end angle=90]--(1,1.3) arc[radius=.3, start angle=270, end angle=360];
	\draw (.3,-.5)--(.3,1);
	\draw (2.7,-.5)--(2.7,1);
	\draw (.7,-.2)--(.7,.7);
	\draw (2.3,-.2)--(2.3,.7);
	\draw (.7,-.2) arc[radius=.3, start angle=180, end angle=270];
	\draw (2.3,-.2) arc[radius=.3,start angle=0, end angle=-90];
	\draw (.7,.7) arc[radius=.3, start angle=180, end angle=90];
	\draw (2.3,.7) arc[radius=.3,start angle=0, end angle=90];
	\draw (2,1)--(1,1);
	\draw (.3,1) arc [radius=.3, start angle=180, end angle=90];
	\draw (.6,1.3)--(1,1.3);
	\draw (1.3,1.6)--(1.3,3);
	\draw (1.3,1.6) arc [radius=.3, start angle=0, end angle=-90];
	\draw (1.7,1.6)--(1.7,3);
	\draw (1.7,1.6) arc [radius=.3, start angle=180, end angle=270];
	\draw (2.7,1) arc [radius=.3, start angle=0, end angle=90];
	\draw (2.4,1.3)--(2,1.3);
	\draw (1,-.5)--(2,-.5);
	\draw [fill=white] (0,0) rectangle (1,.5);
	\draw [fill=white] (2,0) rectangle (3,.5); 
	\draw [fill=white] (1,2) rectangle (2,2.5);
	\node at (0,.25) [left] {$\$$};
	\node at (2,.25) [left] {$\$$};	
	\node at (1,2.25)[left] {$\$$};
	\node at (.5,.25){$e_P$};
	\node at (2.5,.25){$e_Q$};
	\node at (1.5,2.25){$e_Q$};
	\node at (3.5,1.25)[right] {$=$};
	\path [fill=lightgray] (13-+6.3,3)--(13-+6.3,-.5)--(13-+7.7,-.5)--(13-+7.7,1)--(13-+7.3,1)--(13-+7.3,-.1) arc [radius=.2,start angle=180, end angle=270]--(13-+6.9,-.3) arc[radius=.2, start angle=270, end angle=360]--(13-+6.7,1.2) arc [radius=.2,start angle=0, end angle=90]--(13-+7.1,1.4) arc [radius=.2,start angle=90, end angle=180]--(13-+7.3,1.2)--(13-+7.7,1.2)--(13-+7.7,1.4) arc[radius=.2,start angle=180,end angle=90]--(13-+6.9,1.6) arc[radius=.2,start angle=270, end angle=360]--(13-+6.7,3);
	\draw (13-+6.3,3)--(13-+6.3,-.5);
	\draw (13-+6.7,3)--(13-+6.7,1.8);
	\draw (13-+6.7,1.8) arc[radius=.2,start angle=0, end angle=-90];
	\draw (13-+6.9,1.6)--(13-+7.5,1.6);
	\draw (13-+7.5,1.6) arc [radius=.2,start angle=90, end angle=180];
	\draw (13-+7.7,1.4)--(13-+7.7,-.5);
	\draw (13-+7.3,1.2) arc[radius=.2,start angle=180, end angle=90];
	\draw (13-+7.1,1.4)--(13-+6.9,1.4);
	\draw (13-+6.9,1.4) arc [radius=.2, start angle=90, end angle=0];
	\draw (13-+6.7,1.2)--(13-+6.7,-.1);
	\draw (13-+6.7,-.1) arc[radius=.2, start angle=0,end angle=-90];
	\draw (13-+6.9,-.3)--(13-+7.1,-.3);
	\draw (13-+7.1,-.3) arc [radius=.2, start angle=270, end angle=180];
	\draw (13-+7.3,-.1)--(13-+7.3,1.2);
	\draw [fill=white] (13-+6,2) rectangle (13-+7,2.5);
	\draw [fill=white] (13-+7,0) rectangle (13-+8,.5);
	\draw [fill=white] (13-+7,.7) rectangle (13-+8,1.2);
	\node at (13-+7.9,.25)[left]{$\$$};
	\node at (13-+7.9,.95)[left]{$\$$};
	\node at (13-+6.9,2.25)[left]{$\$$};
	\node at (13-+6.5,2.25){$e_Q$};
	\node at (13-+7.5,.25){$e_Q$};
	\node at (13-+7.5,.95){$e_P$};
	\node at (8,1.25)[right]{$=\delta tr(e_Pe_Q)e_Q$};
	\end{tikzpicture}\caption{$e_Q$ is a subprojection of $\frac{1}{\delta tr(e_Pe_Q)}e_P\star e_Q$}
	\label{qe1}
\end{figure}

Thus from the above discussions we conclude $e_P \vee e_Q$ is a subprojection of $\frac{1}{\delta tr(e_Pe_Q)}e_P\star e_Q$ finishing the proof.
\end{proof} 
\begin{proposition}
\label{mini}
Suppose $P,Q$ are distinct minimal intermediate subfactors of a finite index, irreducible subfactor, then 
	\begin{equation}
	\frac{\tau_P\tau_Q}{tr(e_Pe_Q)}\geq \tau_P+\tau_Q-\tau. \label{Equ:keyineq}
	\end{equation}
\end{proposition}
\begin{proof}
	If $P$ and $Q$ are minimal intermediate subfactors, then $P\cap Q=N$. Thus, $e_P\wedge e_Q= e_1.$  Now by Theorem \ref{Thm:Landau} and  Lemma \ref{lem:epeq}, we have 
	\begin{equation*}
	\frac{1}{\delta tr(e_Pe_Q)}e_P * e_Q\geq e_P\vee e_Q.
	\end{equation*}
	Computing the trace of both sides and observing $tr(e_P  *  e_Q)= \delta \tau_P\tau_Q$, we get
	
	\begin{align*}
	\frac{\tau_P\tau_Q}{tr(e_Pe_Q)}&\geq tr(e_P\vee e_Q)\\ 
	&=\tau_P+\tau_Q-tr(e_P\wedge e_Q)\\
	&=\tau_P+\tau_Q-tr(e_1)\\
	&=\tau_P+\tau_Q-\tau.
	\end{align*}
\end{proof}

\begin{theorem}
 \label{bound}
Suppose $P,Q$ are distinct minimal intermediate subfactors of a finite index, irreducible subfactor, then $\alpha(P,Q)> \frac{\pi}{3}$.
\end{theorem}
\begin{proof}
Firstly observe, $(\tau_P+\tau_Q-{\tau})>0$.
	By Equation \eqref{Equ:keyineq}, we have
	\begin{align*}
	tr(e_Pe_Q)&\leq\frac{\tau_P\tau_Q}{\tau_P+\tau_Q-\tau}\\
	tr(e_Pe_Q)-\tau &\leq\frac{\tau_P\tau_Q}{\tau_P+\tau_Q-\tau}-\tau\\
	&=\frac{\tau_P\tau_Q-\tau(\tau_P+\tau_Q)+{\tau}^2}{\tau_P+\tau_Q-{\tau}}\\
	&=\frac{(\tau_P-\tau)(\tau_Q-\tau)}{\tau_P+\tau_Q-\tau}
	\end{align*}
	By Theorem \ref{Thm: alpha-beta},
	\begin{align*}
	\cos(\alpha(P,Q)) &=\frac{tr(e_Pe_Q)-\tau}{\sqrt{\tau_P-\tau}\sqrt{\tau_Q-\tau}}\\
	&\leq\displaystyle\frac{(\tau_P-\tau)^{1/2}(\tau_Q-\tau)^{1/2}}{\tau_P+\tau_Q-\tau} \\
	&<\displaystyle\frac{(\tau_P-\tau)^{1/2}(\tau_Q-\tau)^{1/2}}{\tau_P-\tau+\tau_Q-\tau}\\
	&\leq 1/2	
	\end{align*}
	Therefore, $\alpha(P,Q) > \frac{\pi}{3}$.
\end{proof}

 \section{Number of intermediate subfactors}\label{Sec:number}

In geometry, the kissing number problem asks for the maximum number ${\tau}_n$
of unit spheres that can simultaneously touch the unit sphere in $n$-dimensional
Euclidean space without pairwise overlapping. The value of ${\tau}_n $ is only known for
$n = 1, 2, 3, 4, 8, 24.$ While its determination for $n = 1, 2$ is trivial, it is not the case
for other values of $n.$ The case $n = 3$ was the object of a famous discussion between Isaac Newton
and David Gregory in $1694$. See \cite{Cas} for instance. More generally, a spherical code in dimension $n$ with minimal angular distance $\theta$, is a set of points on the
unit sphere in $\mathbb{R}^n$ with the property that no two points subtend an angle less than
$\theta$ at the origin. Let $A(n,\theta)$ denote the greatest size of such a spherical code. The kissing number problem is then equivalent to the
problem of finding $A(n, \frac{\pi}{3})$. One has the following asymptotic estimate of ${{\tau}_n}$ in \cite{KL} using linear program:
$$ {\tau}_n\leq 2^{0.401n(1+o(1))} = {(1.32042 \cdots)}^{n(1+o(1))}.$$
Upper bound has been independently done by Delsarte, Goethals, and Seidel in  \cite{DGS}.

\begin{theorem}\label{thm:minimal}
	Suppose $N\subset M$ is an finite index, irreducible subfactor. 
Let $\mathcal{L}_m(N,M)$ be the set of all minimal intermediate subfactors of $N\subset M$.
Then the number of minimal intermediate subfactors $|\mathcal{L}_m(N,M)|$ is bounded by the kissing number $\tau_n$, where $n=\dim (N' \cap M_1)$. In particular, 
$$|\mathcal{L}_m(N,M)| < 3^{n}.$$
\end{theorem}

\begin{proof}
Then $\{ v_P: P \in \mathcal{L}_m(N,M)\} $ is a set of unit vectors in $(N^\prime\cap M_1)_{s.a}$, a real inner product space $(N^\prime\cap M_1)_{s.a}$ of dimension $n$. 
Consider the $n$ dimensional unit ball $B_P$ with center at each $2v_P$. Each $B_P$ is adjacent to the unit ball $B(1)$ with center at origin.

By Theorem \ref{bound},  ${\lVert v_P-v_Q\rVert}_2 > 1$ for distinct $P$ and $Q$ in $\mathcal{L}_m(N,M)$. 
So $B_P$ and $B_Q$ are disjoint. Therefore 
$$|\mathcal{L}_m(N,M)|\leq \tau_n.$$

Furthermore, for any $P\in \mathcal{L}_m(N,M)$,
	\begin{equation*}
	B_P\subset \overline{B(3)\setminus B(1)},
	\end{equation*}
where, $B(3)$ stands for the $d$ dimensional ball with center at origin and radius $3$.
	Thus,
	\begin{align*}
	\vert \mathcal{L}_m(N,M)\vert&\leq \displaystyle\frac{Vol(B(3))-Vol(B(1))}{Vol(B(1))}\\
	&=3^d-1
	\end{align*}
\end{proof}

\begin{remark}
For an irreducible subfactor $N\subseteq M$, one has that $\dim (N' \cap M_1)\leq [M:N]$. Thus the number of minimal intermediate subfactors is also bounded by $3^{[M:N]}$.
 \end{remark}

\begin{definition}
	Suppose $\delta^2$ is a real number greater or equal to 2, we define
	\begin{align*}
	I(\delta^2)&=\sup_{N\subset  M}\{\vert Lat(N,M)\vert:N\subset M ~\textrm{is a subfactor with}~[M:N]\leq \delta^2\}\\
	m(\delta^2)&=\sup_{N\subset M}\{\vert \mathcal{L}_m(N,M)\vert:N\subset M ~\textrm{is a subfactor with}~[M:N]\leq \delta^2\}
	\end{align*}
	
\end{definition}

\begin{corollary}\label{cor:minimal}
	Let $\delta^2$ be a real number greater or equal to 2. Then we have 
	\begin{equation*}
	m(\delta^2)\leq 3^{\delta^2}
	\end{equation*}
\end{corollary}

\begin{lemma}\label{lem: mIinequ}
	Suppose $\delta^2\geq 4$, then we have 
	\begin{equation*}
	I(\delta^2)\leq m(\delta^2)I(\delta^2/2)
	\end{equation*}
\end{lemma}
\begin{proof}
	Note that subfactor $R\subset R\rtimes(\mathbb{Z}_2\times\mathbb{Z}_2)$ is of index $4$ and $\mathbb{Z}_2\times\mathbb{Z}_2$ has two non-trivial proper subgroups. Thus $m(\delta^2)\geq 2$ when $\delta^2\geq4$. 
	
	To prove the lemma, we need to show for an arbitrary subfactor $N\subset M$ with $[M:N]\leq\delta^2$, 
	\begin{equation*}
	\vert Lat(N,M) \vert\leq m(\delta^2)I(\delta^2/2)
	\end{equation*}
	
	Case 1: Suppose $\vert \mathcal{L}_m(N,M)\vert=0$, then $\vert Lat(N,M)\vert =2$. (Since in this case $Lat(N,M)=\{N,M\}$.) Note that $m(\delta^2)\geq2$ and $I(\delta^2)\geq2$, and the lemma follows directly.
	\bigskip

	Case 2: Suppose $\vert \mathcal{L}_m(N,M)\vert=1$. Let $P$ be the minimal intermediate subfactor, then we have
	\begin{equation*}
	Lat(N,M)=Lat(P,M)\cup\{N\}
	\end{equation*}
	
	Thus, \begin{align*}
	\vert Lat(N,M)\vert& =\vert Lat(P,M)\vert+1\\
	&\leq I([M:P])+1
	\end{align*}
	
    Since $[M:P]=[M:N]/[P:N]$ and $[P:N]\geq2$, we have $[M:P]\leq[M:N]/2\leq \delta^2/2$. Therefore,
    \begin{align*}
    \vert Lat(N,M)\vert&\leq I(\delta^2/2)+1\\
    &\leq 2I(\delta^2/2)\\
    &\leq m(\delta^2)I(\delta^2/2)    
    \end{align*}
    
    Case 3: Suppose $\vert \mathcal{L}_m(N,M)\vert\geq 2$. It follows that 
    \begin{equation*}
    Lat(N,M)\backslash \{N,M\}\subset \bigcup_{P\in \mathcal{L}_m(N,M)}( Lat(P,M)\backslash{M})
    \end{equation*}
    Therefore,
    \begin{align*}
    \vert Lat(N,M)\vert&\leq \sum_{P\in \mathcal{L}_m(N,M)}( \vert Lat(P,M)\vert-1)+2\\
    &\leq \sum_{P\in \mathcal{L}_m(N,M)}(I([M:P])-1)+2\\
    &\leq \sum_{P\in \mathcal{L}_m(N,M)}(I(\delta^2/2)-1)+2\\
    &\leq \vert \mathcal{L}_m(N,M)\vert I(\delta^2/2)-\vert \mathcal{L}_m(N,M)\vert+2\\
    &\leq \vert \mathcal{L}_m(N,M)\vert I(\delta^2/2)\\
    &\leq m(\delta^2)I(\delta^2/2)    
    \end{align*}
\end{proof}

\begin{theorem}\label{thm:whole}
	Suppose $N\subset M$ is an irreducible subfactor of type II$_1$ of finite index. The number of intermediate subfactors is at most $9^{[M:N]}$.
\end{theorem}
\begin{proof}
	First note that if we have $2\leq [M:N]<4$, then there there are no non-trivial intermediate subfactors for $N\subset M$. Therefore,
	\begin{equation*}
	\vert Lat(N,M) \vert=2<9^2\leq 9^{[M:N]}.
	\end{equation*}
	Suppose $\delta^2=[M:N]\geq 4$, by Lemma \ref{lem: mIinequ}, we have 
	\begin{align*}
	\vert Lat(N,M)\vert&\leq I(\delta^2)\\
	&\leq m(\delta^2)I(\delta^2/2)\\
	&\leq m(\delta^2)m(\delta^2/2)I(\delta^2/2^2)\\
	&\leq m(\delta^2)m(\delta^2/2)m(\delta^2/4)\cdots m(\delta^2/2^k)I(\delta^2/2^{k+1})
	\end{align*}
	where $k$ is the smallest integer such that $2\leq \delta^2/2^{k+1} < 4$.
	
	By Theorem \ref{thm:minimal}, we have
	\begin{align*}
	\vert Lat(N,M)\vert &\leq I(\delta^2/2^{k+1})\prod_{j=0}^{k} 3^{\delta^4/2^j}\\
	&\leq \prod_{j=0}^{k+1} 3^{\delta^2/2^j}~~ (\text{since}~I(\delta^2/2^{k+1})=2<3^{\delta^2/2^{k+1}})\\
	&\leq \prod_{j=0}^{+\infty} 3^{\delta^2/2^j}\\
	&\leq 3^{2\delta^2}=9^{\delta^2}.
	\end{align*}
	This completes the proof.
	
\end{proof}

 \begin{remark}
  Suppose $N\subset M$ is an irreducible subfactor and $N^\prime\cap M_1$ is abelian (for example, $R\subset R\rtimes G$ where G is a finite group acting outerly on $R$. Therefore, Theorem \ref{thm:abSecRc} provides a bound for the cardinality of subgroup of a finite group. ), then  for two distinct minimal intermediate subfactors
  $P$ and $Q$ it is trivial to check that $\alpha_M^N(P,Q)=\frac{\pi}{2}$. Thus the set $\{v_P: P$ is a minimal intermediate subfactor\} forms an orthonormal set and hence
  the number of minimal intermediate subfactors is bounded by $dim(N^{\prime}\cap M_1)  \leq [M:N].$  After that, doing an iteration as above we  obtain a better bound than
 Proposition 5.1 in \cite{TW} for the cardinality of the lattice $\mathcal{L}(N\subset M)$ as explained below.
 \end{remark}

\begin{theorem}\label{thm:abSecRc}
	Suppose $N\subset M$ is an irreducible subfactor and $N^\prime\cap M_1$ is abelian, then 
	\begin{equation*}
	\vert Lat(N,M)\vert\leq (\frac{[M:N]}{\sqrt{2}})^{\frac{\log([M:N])}{2}}
	\end{equation*}
\end{theorem}
\begin{proof}
	Let $P$ and $Q$ be two minimal intermediate subfactors. Then we have 
	\begin{align*}
	\cos(\alpha(P,Q))&=\frac{tr((e_P-e_1)(e_Q-e_1))}{{\lVert e_P-e_1\rVert}_2{\lVert e_Q-e_1\rVert}_2}\\
	&=\frac{tr(e_Pe_Q)-tr(e_1)}{{\lVert e_P-e_1\rVert}_2{\lVert e_Q-e_1\rVert}_2}
	\end{align*}
	Note that $e_P,e_Q\in N^\prime\cap M_1$, which is abelian and $P\cap Q=N$. Thus, we have
	\begin{align*}
	tr(e_Pe_Q)&=tr(e_P\wedge e_Q)=tr(e_1).\\
	\Rightarrow&\cos(\alpha(P,Q))=0.
	\end{align*}
	Therefore, for any two minimal intermediate subfactors $P$ and $Q$, $\alpha(P,Q)=\pi/2$. In particular, this means that the set $\{v_P:\text{ P is a minimal intermediate subfactor}\}$ is an orthonormal set. Therefore, 
	\begin{equation*}
	\vert \mathcal{L}_m(N,M)\vert\leq \dim (N^\prime\cap M_1)\leq [M:N]
	\end{equation*}
	This implies that $m(\delta^2)\leq \delta^2$ and therefore, by Lemma \ref{lem: mIinequ}, we have
	\begin{align*}
	\vert Lat(N,M)\vert&\leq I(\delta^2)\\
	&\leq m(\delta^2)I(\delta^2/2)\\
	&\leq m(\delta^2)m(\delta^2/2)I(\delta^2/2^2)\\
	&\leq m(\delta^2)m(\delta^2/2)m(\delta^2/4)\cdots m(\delta^2/2^k)I(\delta^2/2^{k+1})
	\end{align*}
	where $k$ is the smallest integer such that $2\leq \delta^2/2^{k+1} < 4$, i.e, $k+1\leq\log(\delta^2/2)$. Note that $m(\delta^2)\leq \delta^2$, we have
	\begin{align*}
	\vert Lat(N,M)\vert&\leq m(\delta^2)m(\delta^2/2)m(\delta^2/4)\cdots m(\delta^2/2^k)I(\delta^2/2^{k+1})\\
	&\leq I(\delta^2/2^{k+1}) \displaystyle\prod_{j=0}^{k} \delta^2/2^j\\
	&\leq \displaystyle\prod_{j=0}^{k+1}\delta^2/2^j\\
	&=(\delta^2)^{k+1}\displaystyle\frac{1}{2^{(k+1)(k+2)/2}}\\
	&\leq (\delta^2)^{k+1}\displaystyle\frac{1}{2^{(k+1)/2}}\\
	&=(\displaystyle\frac{\delta^2}{\sqrt{2}})^{k+1}\\
	&\leq (\frac{\delta^2}{\sqrt{2}})^{\log(\delta^2/2)}
	\end{align*}
	This completes the proof.
\end{proof}

\bibliography{bibliography}
\bibliographystyle{amsalpha}
\end{document}